\documentclass[reqno,openbib,runningheads,a4paper,12pt]{amsart}%
\usepackage{graphicx}
\usepackage{amssymb}
\usepackage{amsfonts}
\usepackage{amsmath}
\usepackage{graphicx}
\usepackage{lineno}
\usepackage[authoryear]{natbib}
\usepackage[dvipsnames]{xcolor}
\usepackage{epsf}
\usepackage{lscape}
\usepackage[legalpaper,bookmarks=true,colorlinks=true,linkcolor=blue,citecolor=blue]%
{hyperref}%
\providecommand{\U}[1]{\protect\rule{.1in}{.1in}}
\providecommand{\U}[1]{\protect \rule{.1in}{.1in}}
\newtheorem{theorem}{Theorem}[section]

\newtheorem{proposition}{Proposition}[section]

\newtheorem{lemma}{Lemma}[section]

\makeatletter
\renewcommand{\@biblabel}[1]{}
\makeatother
\begin{document}

\begin{center}
{\Large Robust Tail Index Estimation under Random Censoring via Minimum
Density Power Divergence}\bigskip

{\large Nour Elhouda Guesmia, Abdelhakim Necir}$^{\ast}${\large , Djamel
Meraghni}\medskip\newline

{\small \textit{Laboratory of Applied Mathematics, Mohamed Khider University,
Biskra, Algeria}}\medskip\medskip
\end{center}

\noindent\textbf{Abstract}\smallskip

\noindent We propose a robust estimator for the tail index of Pareto-type
distributions under random right-censoring, constructed within the minimum
density power divergence (MDPD) framework and based on the Nelson--Aalen
estimator of the cumulative hazard function. To our knowledge, this is the
first application of the MDPD methodology to tail index estimation in the
presence of random censoring. Under mild regularity conditions and within the
weak censoring regime, the estimator is shown to be consistent and
asymptotically normal. Its finite-sample performance is assessed through Monte
Carlo simulations, revealing improved robustness--efficiency trade-offs
compared to standard non-robust tail index estimators. Robustness is
investigated under both pre-censoring and post-censoring contamination
schemes. While pre-censoring contamination provides a meaningful framework for
robustness assessment, post-censoring contamination directly alters the
observable data and highlights the sensitivity of reconstruction-based
approaches. The practical relevance of the method is illustrated using an
insurance claims dataset with light censoring and fully observable extremes.
An additional application to AIDS survival data is included for illustrative
purposes, emphasizing the challenges encountered under stronger
censoring.\medskip

\medskip


\vfill

\vfill

\noindent{\small $^{\text{*}}$Corresponding author:
ah.necir@univ-biskra.dz\newline\noindent\textit{E-mail address:}\newline
nourelhouda.guesmia@univ-biskra.dz\texttt{\ }(N.~Guesmia)\newline
djamel.meraghni@univ-biskra.dz (D.~Meraghni)}\newline

\section{\textbf{Introduction\label{sec1}}}

\noindent Right-censored data pose a common challenge in statistics, as values
are only known up to a certain threshold, while those beyond it remain
unobserved. This situation frequently arises in fields such as survival
analysis, reliability engineering, and insurance. Pareto-type distributions
are commonly used to model data exhibiting extreme values. These distributions
are characterized by a relatively high probability of producing observations
far from the mean, in contrast with the normal distribution. They are
particularly relevant in applications such as insurance, finance, and
environmental studies, where rare events can have a substantial
impact.\smallskip\ 

\noindent Let $X_{1},X_{2},...,X_{n}$ be a sample of size $n\geq1$ from a
random variable (rv) $X,$ and let $C_{1},C_{2},...,C_{n}$ be another sample
from a rv $C,$ both defined on a probability space $\left(  \Omega
,\mathcal{A},\mathbf{P}\right)  ,$ with continuous cumulative distribution
functions (cdfs) $F$ and $G,$ respectively. Assume that $X$\ and $C$ are
independent. We also suppose that $X$ is right censored by $C,$ meaning that
for each index $1\leq j\leq n,$ we can only observe the variable
\[
Z_{j}:=\min\left\{  X_{j},C_{j}\right\}
\]
and the indicator variable
\[
\delta_{j}:=\mathbb{I}_{\left\{  X_{j}\leq C_{j}\right\}  },
\]
which determines whether or not $X$ has been observed. We assume that the tail
functions $\overline{F}:=1-F$ and $\overline{G}:=1-G$ are regularly varying at
infinity (or Pareto-type) with positive tail indices $\gamma_{1}>0$ and
$\gamma_{2}>0,$ allowing for slowly varying deviations from the strict Pareto
form. More precisely, for any $x>0,$%
\begin{equation}
\lim_{u\rightarrow\infty}\frac{\overline{F}(ux)}{\overline{F}(u)}%
=x^{-1/\gamma_{1}}\text{ and }\lim_{u\rightarrow\infty}\frac{\overline{G}%
(ux)}{\overline{G}(u)}=x^{-1/\gamma_{2}}.\label{RV}%
\end{equation}
Equivalently, these tail behaviors can be expressed in terms of slowly varying
functions $\ell_{1}\left(  x\right)  $ and $\ell_{2}\left(  x\right)  $ to
account for deviations from strict Pareto distributions:%
\begin{equation}
\overline{F}(x)=x^{-1/\gamma_{1}}\ell_{1}\left(  x\right)  \text{ and
}\overline{G}(x)=x^{-1/\gamma_{2}}\ell_{2}\left(  x\right)  ,\label{RVbis}%
\end{equation}
where $\ell_{1}\left(  x\right)  $ and $\ell_{2}\left(  x\right)  $ vary
slowly at infinity, i.e., for any fixed $t>0,$
\[
\ell_{j}\left(  tx\right)  /\ell_{j}\left(  x\right)  \rightarrow1,\text{ as
}x\rightarrow\infty,\text{ }j=1,2.
\]
These expressions make explicit the Pareto-type form of the distributions and
clarify that the limit relations correspond to first-order asymptotics, while
the slowly varying functions explicitly capture deviations from a strict
Pareto model.\medskip

\noindent We denote the cdf of $Z$ by $H.$ Then, using the independence of $X$
and $C,$ we have
\[
\overline{H}=\overline{F}\times\overline{G},
\]
which implies that $\overline{H}$ is also regularly varying at infinity, with
tail index
\[
\gamma:=\frac{\gamma_{1}\gamma_{2}}{\gamma_{1}+\gamma_{2}}.
\]
This characterization of the tail behavior of the observed variable $Z$
provides a firm basis for estimating the tail index $\gamma$ under right
censoring. In the subsequent sections, we develop estimation procedures that
appropriately account for censored observations and examine their asymptotic
properties, including robustness to contamination.\medskip

\noindent In the presence of extreme values and right-censored data, various
methods have been proposed for estimating the tail index. Several techniques
have been developed to address the specific challenges posed by such data.
Many studies have focused on modifying traditional tail index estimation
methods, notably Hill's estimator \citep[][]{Hill75} to accommodate censored
observations. In particular, \cite{EnFG08} adapted this estimator to handle
right-censored data, leading to the following form:
\[
\widehat{\gamma}_{1}^{\left(  \mathrm{EFG}\right)  }:=\frac{\widehat{\gamma
}^{\left(  H\right)  }}{\widehat{p}},
\]
where%
\[
\widehat{\gamma}^{\left(  \mathrm{H}\right)  }:=k^{-1}%
{\displaystyle\sum\limits_{i=1}^{k}}
\log\left(  Z_{n-i+1:n}/Z_{n-k:n}\right)  ,
\]
denotes the classical Hill estimator corresponding to the tail index $\gamma,$
and
\[
\widehat{p}:=k^{-1}%
{\displaystyle\sum_{i=1}^{k}}
\delta_{\left[  n-i+1:n\right]  },
\]
represents an estimator of the proportion of upper non-censored observations,
\begin{equation}
p:=\frac{\gamma_{2}}{\gamma_{1}+\gamma_{2}}.\label{p}%
\end{equation}
The integer sequence $k=k_{n}$ represents the number of top order statistics
used in the estimation of the tail index, such that $k\rightarrow\infty$ and
$k/n\rightarrow0$ as $n\rightarrow\infty.$ This ensures that enough extreme
observations are included for reliable estimation, while avoiding the
inclusion of too many moderate values that could bias the tail index
estimation.\smallskip

\noindent The sequence of rvs $Z_{1:n}\leq...\leq Z_{n:n}$ represents the
order statistics pertaining to the sample $Z_{1},...,Z_{n}$ and $\delta
_{\left[  1:n\right]  },...,\delta_{\left[  n:n\right]  }$ denotes the
corresponding concomitant values, satisfying $\delta_{\left[  j:n\right]
}=\delta_{i}$ for $i$ such that $Z_{j:n}=Z_{i}.$ This notation clarifies the
alignment between the ordered observations and their censoring indicators,
which is essential for constructing the censored Hill-type
estimator.\smallskip

\noindent On the basis of a Kaplan-Meier integration, \cite{WW2014} proposed a
consistent estimator to the tail index $\gamma_{1}$ defined by
\[
\widehat{\gamma}_{1}^{\left(  \mathrm{W}\right)  }:=%
{\displaystyle\sum\limits_{i=1}^{k}}
\frac{\overline{F}_{n}^{\left(  KM\right)  }\left(  Z_{n-i:n}\right)
}{\overline{F}_{n}^{\left(  KM\right)  }\left(  Z_{n-k:n}\right)  }\log
\frac{Z_{n-i+1:n}}{Z_{n-i:n}},
\]
where
\[
F_{n}^{\left(  \mathrm{KM}\right)  }\left(  x\right)  :=\underset{Z_{i:n}\leq
x}{1-%
{\displaystyle\prod}
}\left(  \dfrac{n-i}{n-i+1}\right)  ^{\delta_{\left[  i:n\right]  }},
\]
denotes the popular Kaplan-Meier estimator of the cdf $F$ \citep[] []{KM58}.
The asymptotic normality of $\widehat{\gamma}_{1}^{\left(  \mathrm{W}\right)
}$ is established in \cite{BWW2019} by considering Hall's model
\citep[] []{Hall82}. Bias reduction in tail index estimation under censoring
has been addressed in \cite{BBWG2016}, \cite{BMV2018} and \cite{BWW2019}%
.\smallskip

\noindent By using a Nelson-Aalen integration, more recently \cite{MNS2025}
derived a new estimator for $\gamma_{1}$ given by%
\[
\widehat{\gamma}_{1}^{\left(  \mathrm{MNS}\right)  }:=%
{\displaystyle\sum\limits_{i=1}^{k}}
\frac{\delta_{\left[  n-i+1:n\right]  }}{i}\frac{\overline{F}_{n}^{\left(
NA\right)  }\left(  Z_{n-i+1:n}\right)  }{\overline{F}_{n}^{\left(  NA\right)
}\left(  Z_{n-k:n}\right)  }\log\frac{Z_{n-i+1:n}}{Z_{n-k:n}},
\]
where
\begin{equation}
F_{n}^{\left(  \mathrm{NA}\right)  }\left(  x\right)  =1-\prod_{Z_{i:n}<x}%
\exp\left\{  -\frac{\delta_{\left[  i:n\right]  }}{n-i+1}\right\}
,\label{Nelson}%
\end{equation}
denotes the well-known Nelson-Aalen estimator of the cdf $F$
\citep[] []{Nelson1972}. Note that
\begin{equation}
\frac{\overline{F}_{n}^{\left(  \mathrm{NA}\right)  }\left(  Z_{n-i+1:n}%
\right)  }{\overline{F}_{n}^{\left(  \mathrm{NA}\right)  }\left(
Z_{n-k:n}\right)  }=\prod_{j=i+1}^{k}\exp\left\{  -\frac{\delta_{\left[
n-j+1:n\right]  }}{j}\right\}  ,\label{ratio}%
\end{equation}
thus the formula of $\widehat{\gamma}_{1}^{\left(  \mathrm{MNS}\right)  }$ can
be rewritten as
\[
\widehat{\gamma}_{1}^{\left(  \mathrm{MNS}\right)  }:=%
{\displaystyle\sum\limits_{i=1}^{k}}
a_{ik}\log\frac{Z_{n-i+1:n}}{Z_{n-k:n}},
\]
where
\begin{equation}
a_{ik}:=\frac{\delta_{\left[  n-i+1:n\right]  }}{i}\prod_{j=i+1}^{k}%
\exp\left\{  -\frac{\delta_{\left[  n-j+1:n\right]  }}{j}\right\}
.\label{aik}%
\end{equation}
This representation highlights the connection between the Nelson--Aalen
weights $a_{ik}$ and the Kaplan--Meier-based integration, and makes explicit
the recursive weighting scheme that adjusts for censoring in the upper
tail.\smallskip

\noindent The asymptotic properties of both the Kaplan-Meier-based estimator
$\widehat{\gamma}_{1}^{\left(  \mathrm{W}\right)  }$ and the
Nelson-Aalen-based estimator $\widehat{\gamma}_{1}^{\left(  \mathrm{MNS}%
\right)  }$ are established provided that $p>1/2.$ This condition ensures that
a sufficiently large fraction of extreme observations remains uncensored,
which is essential for reliable estimation of $\gamma_{1}.$ Equivalently, it
can be expressed as $\gamma_{1}<\gamma_{2},$ meaning that the censoring
distribution has a lighter tail than the distribution of interest. If
$p\leq1/2$ (or $\gamma_{1}\geq\gamma_{2}$)$,$ the effective tail sample size
may be too small, and the Gaussian approximation underlying these estimators
may break down. In practice, this condition serves as a natural guideline for
the applicability of both estimators, emphasizing the importance of the
relative tail heaviness between the observed and censoring
distributions.\smallskip

\noindent The authors showed that the two tail index estimators
$\widehat{\gamma}_{1}^{\left(  \mathrm{W}\right)  }$ and $\widehat{\gamma}%
_{1}^{\left(  \mathrm{MNS}\right)  }$ exhibit similar performances in terms of
both bias and mean squared error (MSE). Indeed, studies comparing the
Kaplan--Meier and Nelson--Aalen estimators have shown that the latter exhibits
an almost identical statistical behavior; see, for instance,
\cite{Colosimo2002}. On the other hand, establishing the asymptotic properties
of extreme Kaplan--Meier integrals poses certain technical difficulties. To
overcome this issue, \cite{MNS2025} introduced a Nelson--Aalen tail-product
limit process and established its Gaussian approximation in terms of standard
Wiener processes. Using this approach, they proved the consistency and
asymptotic normality of the proposed estimator $\widehat{\gamma}_{1}^{\left(
\mathrm{MNS}\right)  }$ under the first- and second-order regular variation
conditions, namely the assumptions $\left(  \ref{RV}\right)  $ and $\left(
\ref{second-order}\right)  $ respectively. As expected, it was also shown that
both the asymptotic bias and variance of $\widehat{\gamma}_{1}^{\left(
\mathrm{W}\right)  }$ and $\widehat{\gamma}_{1}^{\left(  \mathrm{MNS}\right)
}$ are equal, highlighting the practical equivalence of the two estimators
under the stated assumptions.\smallskip

\noindent The most common method for estimating the parameters of an extreme
value distribution in extreme value analysis relies on maximum likelihood
estimation (MLE) method. Although these estimators possess desirable
asymptotic properties, they can be sensitive to outlying observations from the
assumed extreme value models; see, for instance, \cite{BS2000}. Consequently,
robust statistical methods provide a better alternative to mitigate the
influence of outliers and deviations from the underlying parametric models. It
has been shown that employing robust statistical ideas in extreme value theory
improves the quality and precision of estimates \citep[][]{DE2006}.\smallskip

\noindent\cite{JS2004} appear to be the first authors to employ the minimum
density power divergence (MDPD) of \cite{Basu98} for the robust estimation of
the parameters of an extreme value distribution. Since then, this divergence
measure has become a widely used tool for robust estimation of parameters of
extreme value distributions. \cite{Kim2008}, \cite{DGG13}, \cite{GGV2014}, and
\cite{DGG2021} have applied the MDPD approach to estimate the tail index and
quantiles from Pareto-type distributions. Recently, \cite{Ghosh2017} proposed
a robust MDPD estimator for a real-valued tail index. This estimator is a
robust generalization of the estimator proposed by \cite{MB2003} and addresses
non-identical distributions in the exponential regression model using the
approach in \cite{GB2013}. Furthermore, \cite{DGG13} employed the MDPD concept
on an extended Pareto distribution for relative excesses over a high
threshold. \cite{MWG2023A} and \cite{MWGY2023B} recently proposed a robust
estimator for the tail index of a Pareto-type distribution using the MDPD
approach applied to an exponential regression model.\smallskip

\noindent While robust estimation techniques have been extensively developed
for complete data settings, their extension to the context of randomly
right-censored observations remains comparatively less explored. Such
censoring mechanisms are prevalent in survival analysis, reliability theory,
and actuarial applications, where incomplete information introduces additional
statistical challenges. Classical estimators of the extreme value index under
censoring, such as Hill-type approaches or those based on Kaplan-Meier and
Nelson-Aalen integrals, often remain sensitive to outliers and model
deviations, thereby limiting their practical robustness in real-world
applications.\textbf{\smallskip}

\noindent To the best of our knowledge, apart from the recent work of
\cite{DGG2021}, there exists no robust estimator specifically tailored to the
tail index under random censoring. Their approach focuses on the conditional
tail index, employing kernel smoothing techniques and local threshold
selection for each covariate value. While the estimator effectively reduces
bias and variance in the presence of contamination, it relies on a separate
selection of tuning parameters and is restricted to the conditional setting,
which complicates its direct application to the unconditional tail index. This
limitation underscores the need for a robust and flexible framework capable of
handling both unconditional and censored extreme value scenarios, motivating
the development of the MDPD-based estimator proposed in this work.\smallskip

\noindent The remainder of this paper is organized as follows. Section
\ref{sec2} provides a brief overview of the MDPD estimation method, originally
introduced by \cite{Basu98}, highlighting its robust properties in the
presence of outliers. Section \ref{sec3} discusses in detail the asymptotic
properties (consistency and asymptotic normality) of the proposed estimator,
with full proofs deferred to Section \ref{sec6}. The finite-sample performance
of the estimator is examined in Section \ref{sec4}, through a comprehensive
simulation study, including comparisons with existing tail index estimators
under censoring and contamination. Section \ref{sec5} presents real-data
applications, including a dataset on insurance claims (where censoring is
relatively light, $p>1/2$) and the classical AIDS survival dataset,
illustrating the practical relevance and robustness of the proposed
methodology.\smallskip

\noindent Finally, Appendix A compiles several technical lemmas and
propositions that are instrumental for the theoretical development, while
Appendix B contains the figures and additional results related to the
simulation study.

\section{\textbf{Minimum density power divergence\label{sec2}}}

\noindent\cite{Basu98} introduced a new measure of divergence between two
probability densities $\ell$ and $f,$ known as the density power divergence
(DPD), defined by
\begin{equation}
d_{\alpha}\left(  f,\ell\right)  :=\left\{
\begin{array}
[c]{ll}%
\int_{\mathbb{R}}\left[  \ell^{1+\alpha}\left(  x\right)  -\left(  1+\dfrac
{1}{\alpha}\right)  \ell^{\alpha}\left(  x\right)  f\left(  x\right)
+\dfrac{1}{\alpha}f^{1+\alpha}\left(  x\right)  \right]  dx, & \alpha
>0\medskip\\
\int_{\mathbb{R}}f\left(  x\right)  \log\dfrac{f\left(  x\right)  }%
{\ell\left(  x\right)  }dx, & \alpha=0,
\end{array}
\right.  \label{d}%
\end{equation}
where $\alpha\geq0$ is a tuning parameter controlling the trade-off between
efficiency and robustness. The case $\alpha=0$ is the limit of the general
expression $\left(  \alpha>0\right)  $ as $\alpha\downarrow0$ yielding the
classical Kullback-Leibler divergence\textbf{ }$d_{0}\left(  f,\ell\right)  .$
Let us consider a parametric family of densities $\left\{  \ell_{\theta
}:\Theta\subset\mathbb{R}^{p}\right\}  $ and suppose that the goal is to
estimate the parameter $\theta.$ Let $F$ denote the cdf corresponding to the
density $f.$ The MDPD functional $T_{\alpha}\left(  F\right)  $ is defined as
\[
d_{\alpha}\left(  f,\ell_{T_{\alpha}\left(  F\right)  }\right)  =\min
_{\theta\in\Theta}d_{\alpha}\left(  f,\ell_{\theta}\right)  .
\]
The term $\int f^{1+\alpha}\left(  x\right)  dx$ in $\left(  \ref{d}\right)  $
the divergence does not depend on $\theta,$ and thus does not affect the
minimization. Therefore, minimizing $d_{\alpha}\left(  f,\ell_{\theta}\right)
$ over $\theta\in\Theta$ reduces to minimizing
\[
\delta_{\alpha}\left(  f;\theta\right)  :=\left\{
\begin{array}
[c]{lc}%
\int_{\mathbb{R}}\ell_{\theta}^{1+\alpha}\left(  x\right)  dx-\left(
1+\dfrac{1}{\alpha}\right)  \int_{\mathbb{R}}\ell_{\theta}^{\alpha}\left(
x\right)  dF\left(  x\right)  , & \alpha>0,\medskip\\
-\int_{\mathbb{R}}\log\ell_{\theta}\left(  x\right)  dF\left(  x\right)  , &
\alpha=0.
\end{array}
\right.
\]
Given a sample $X_{1},...,X_{n}$ from the cdf $F,$ we estimate $\delta
_{\alpha}\left(  f;\theta\right)  $ by replacing $F$ with its empirical
counterpart $F_{n}\left(  x\right)  :=n^{-1}\sum_{i=1}^{n}\mathbb{I}_{\left\{
X_{i}\leq x\right\}  }.$ The MDPD estimator is then the minimizer (over
$\theta\in\Theta)$ of the empirical objective function
\[
L_{n.\alpha}\left(  \theta\right)  =\left\{
\begin{array}
[c]{lc}%
\int_{\mathbb{R}}\ell_{\theta}^{1+\alpha}\left(  x\right)  dx-\left(
1+\dfrac{1}{\alpha}\right)  \dfrac{1}{n}%
{\displaystyle\sum\limits_{i=1}^{n}}
\ell_{\theta}^{\alpha}\left(  X_{i}\right)  , & \alpha>0,\medskip\\
-\dfrac{1}{n}%
{\displaystyle\sum\limits_{i=1}^{n}}
\log\ell_{\theta}\left(  X_{i}\right)  , & \alpha=0.
\end{array}
\right.  .
\]
That is, $\widehat{\theta}_{n,\alpha}:=\arg\min_{\theta\in\Theta}L_{n.\alpha
}\left(  \theta\right)  ,$ which satisfies the estimating equation
\[
\left\{
\begin{array}
[c]{lc}%
\int_{\mathbb{R}}\dfrac{d}{d\theta}\ell_{\theta}^{1+\alpha}\left(  x\right)
dx-\left(  1+\dfrac{1}{\alpha}\right)  \dfrac{1}{n}%
{\displaystyle\sum\limits_{i=1}^{n}}
\dfrac{d}{d\theta}\ell_{\theta}^{\alpha}\left(  X_{i}\right)  =0, &
\alpha>0,\medskip\\
\dfrac{1}{n}%
{\displaystyle\sum\limits_{i=1}^{n}}
\dfrac{d}{d\theta}\log\ell_{\theta}\left(  X_{i}\right)  =0, & \alpha=0.
\end{array}
\right.
\]
Small $\alpha$ yields high efficiency with limited robustness, while moderate
$\alpha$ values provide strong robustness under contamination with minor
efficiency loss compared to MLE under correct model specification.

\subsection{MDPD estimator of $\gamma_{1}$ for right-censored data}

\noindent For a fixed threshold $u>0,$ consider the relative excess rv
$Y:=X/u$ conditional on $X>u.$ Its cdf is given by
\[
F_{u}\left(  x\right)  :=1-\frac{\overline{F}\left(  ux\right)  }{\overline
{F}\left(  u\right)  },
\]
with corresponding probability density function $f_{u}.$ Assuming that
$\overline{F}\in\mathcal{RV}_{\left(  -1/\gamma_{1}\right)  },$ it follows
that $F_{u}\left(  x\right)  \approx1-x^{-1/\gamma_{1}},$ as $u\rightarrow
\infty.$ Therefore, the rv $Y$ may be approximated by the Pareto distribution
with density
\[
\ell_{\gamma_{1}}\left(  x\right)  :=\frac{d}{dx}\left(  1-x^{-1/\gamma_{1}%
}\right)  =\gamma_{1}^{-1}x^{-1-1/\gamma_{1}},\text{ for }x\geq1.
\]
The corresponding DPD objective function is
\begin{equation}
\mathcal{L}_{u,\alpha}^{\ast}\left(  \gamma_{1}\right)  :=\left\{
\begin{array}
[c]{lc}%
\int_{1}^{\infty}\ell_{\gamma_{1}}^{1+\alpha}\left(  x\right)  dx-\left(
1+\dfrac{1}{\alpha}\right)  \int_{u}^{\infty}\ell_{\gamma_{1}}^{\alpha}\left(
x/u\right)  d\dfrac{F\left(  x\right)  }{\overline{F}\left(  u\right)  }, &
\alpha>0,\medskip\\
\int_{u}^{\infty}\log\ell_{\gamma_{1}}\left(  x/u\right)  d\dfrac{F\left(
x\right)  }{\overline{F}\left(  u\right)  }, & \alpha=0,
\end{array}
\right.  \label{L}%
\end{equation}
Letting $u=Z_{n-k:n}$ and replacing $F$ with the Nelson-Aalen estimator
$F_{n}^{\left(  NA\right)  }$ gives the empirical objective function\textbf{ }%
\begin{equation}
\mathcal{L}_{k,\alpha}^{\ast}\left(  \gamma_{1}\right)  :=\left\{
\begin{array}
[c]{lc}%
\int_{1}^{\infty}\ell_{\gamma_{1}}^{1+\alpha}\left(  x\right)  dx-\left(
1+\dfrac{1}{\alpha}\right)  \int_{1}^{\infty}\ell_{\gamma_{1}}^{\alpha}\left(
x/Z_{n-k:n}\right)  d\dfrac{F_{n}^{\left(  NA\right)  }\left(  x\right)
}{\overline{F}_{n}^{\left(  NA\right)  }\left(  Z_{n-k:n}\right)  }, &
\alpha>0,\medskip\\
\int_{1}^{\infty}\log\ell_{\gamma_{1}}\left(  x/Z_{n-k:n}\right)
d\dfrac{F_{n}^{\left(  NA\right)  }\left(  x\right)  }{\overline{F}%
_{n}^{\left(  NA\right)  }\left(  Z_{n-k:n}\right)  }, & \alpha=0.
\end{array}
\right.  \label{emp-L}%
\end{equation}
The MDPD tail index estimator under random censoring, denoted $\widehat{\gamma
}_{1,\alpha},$ is then obtained by minimizing $\mathcal{L}_{k,\alpha}^{\ast
}\left(  \gamma_{1}\right)  .$ That is, by solving the following equation:
\[
\left\{
\begin{array}
[c]{lc}%
\begin{array}
[c]{l}%
\int_{1}^{\infty}\dfrac{d}{d\gamma_{1}}\ell_{\gamma_{1}}^{1+\alpha}\left(
\dfrac{x}{Z_{n-k:n}}\right)  dx\\
\ \ \ \ \ \ \ \ \ \ \ \ \ \ \ \ -\left(  1+\dfrac{1}{\alpha}\right)  \int%
_{1}^{\infty}\dfrac{d}{d\gamma_{1}}\ell_{\gamma_{1}}^{\alpha}\left(  \dfrac
{x}{Z_{n-k:n}}\right)  d\dfrac{F_{n}^{\left(  NA\right)  }\left(  x\right)
}{\overline{F}_{n}^{\left(  NA\right)  }\left(  Z_{n-k:n}\right)  }=0,
\end{array}
& \alpha>0,\medskip\\
\int_{1}^{\infty}\dfrac{d}{d\gamma_{1}}\log\ell_{\gamma_{1}}\left(  \dfrac
{x}{Z_{n-k:n}}\right)  d\dfrac{F_{n}^{\left(  NA\right)  }\left(  x\right)
}{\overline{F}_{n}^{\left(  NA\right)  }\left(  Z_{n-k:n}\right)  }=0, &
\alpha=0.
\end{array}
\right.
\]
By elementary calculations:%
\[
\dfrac{d}{d\gamma_{1}}\ell_{\gamma_{1}}^{\alpha}\left(  x\right)
=\frac{\alpha}{\gamma_{1}^{2+\alpha}}\frac{\log x-\gamma_{1}}{x^{\alpha\left(
1+1/\gamma_{1}\right)  }},
\]%
\[
\dfrac{d}{d\gamma_{1}}\log\ell_{\gamma_{1}}\left(  x\right)  =\dfrac
{d}{d\gamma_{1}}\log\left(  \gamma_{1}^{-1}x^{-1-1/\gamma_{1}}\right)
=-\frac{\gamma_{1}-\log x}{\gamma_{1}^{2}}%
\]
and%
\[
\int_{1}^{\infty}\dfrac{d}{d\gamma_{1}}\ell_{\gamma_{1}}^{1+\alpha}\left(
x\right)  dx=-\frac{\alpha\left(  \alpha+1\right)  \left(  \gamma
_{1}+1\right)  }{\gamma_{1}^{\alpha+1}\left(  \alpha\left(  1+\gamma
_{1}\right)  +1\right)  ^{2}}.
\]
Hence, the estimating equation for $\widehat{\gamma}_{1,\alpha}$ is
\begin{equation}
\left\{
\begin{array}
[c]{lc}%
{\displaystyle\int_{Z_{n-k:n}}^{\infty}}
\dfrac{\gamma_{1}-\log\left(  x/Z_{n-k:n}\right)  }{\left(  x/Z_{n-k:n}%
\right)  ^{\alpha\left(  1+1/\gamma_{1}\right)  }}d\dfrac{F_{n}^{\left(
NA\right)  }\left(  x\right)  }{\overline{F}_{n}^{\left(  NA\right)  }\left(
Z_{n-k:n}\right)  }=\dfrac{\alpha\gamma_{1}\left(  \gamma_{1}+1\right)
}{\left(  1+\alpha+\alpha\gamma_{1}\right)  ^{2}}, & \alpha>0,\medskip\\%
{\displaystyle\int_{Z_{n-k:n}}^{\infty}}
\left(  \gamma_{1}-\log\dfrac{x}{Z_{n-k:n}}\right)  d\dfrac{F_{n}^{\left(
NA\right)  }\left(  x\right)  }{\overline{F}_{n}^{\left(  NA\right)  }\left(
Z_{n-k:n}\right)  }=0, & \alpha=0.
\end{array}
\right.  \label{int-formula}%
\end{equation}
\cite{MNS2025} stated that
\begin{equation}
dF_{n}^{\left(  \mathrm{NA}\right)  }\left(  x\right)  =\frac{\overline{F}%
_{n}^{\left(  \mathrm{NA}\right)  }\left(  x\right)  dH_{n}^{\left(  1\right)
}\left(  x\right)  }{\overline{H}_{n}\left(  x^{-}\right)  },\label{dFn}%
\end{equation}
where $H_{n}$ and $H_{n}^{\left(  1\right)  }$ are the empirical counterparts
of the cdf $H$ and the sub-distribution $H^{\left(  1\right)  },$ respectively
defined by:%
\[
H_{n}\left(  x\right)  :=\frac{1}{n}\sum_{i=1}^{n}\mathbb{I}_{\left\{
Z_{i:n}\leq x\right\}  }\text{ and }H_{n}^{\left(  1\right)  }\left(
x\right)  :=\frac{1}{n}\sum_{i=1}^{n}\delta_{\left[  i:n\right]  }%
\mathbb{I}_{\left\{  Z_{i:n}\leq x\right\}  },
\]
(see, for example, \cite{SW86} page 294). Here, $f\left(  x^{-}\right)  $
denotes the left limit, at $x,$ of the function $f.$ By substituting
$dF_{n}^{\left(  NA\right)  }\left(  x\right)  $ with its expression in
equation $\left(  \ref{int-formula}\right)  ,$ we get%
\begin{equation}
\left\{
\begin{array}
[c]{lc}%
{\displaystyle\sum\limits_{i=1}^{k}}
a_{ik}\dfrac{\gamma_{1}-\log\left(  Z_{n-i+1:n}/Z_{n-k:n}\right)  }{\left(
Z_{n-i+1:n}/Z_{n-k:n}\right)  ^{\alpha\left(  1+1/\gamma_{1}\right)  }}%
=\dfrac{\alpha\gamma_{1}\left(  \gamma_{1}+1\right)  }{\left(  1+\alpha
+\alpha\gamma_{1}\right)  ^{2}}, & \alpha>0,\medskip\\%
{\displaystyle\sum\limits_{i=1}^{k}}
a_{ik}\left\{  \gamma_{1}-\log\dfrac{Z_{n-i+1:n}}{Z_{n-k:n}}\right\}  =0, &
\alpha=0,
\end{array}
\right.  \label{estim-equa}%
\end{equation}
where $a_{ik}$ are the Nelson-Aalen weights defined earlier in $\left(
\ref{aik}\right)  .$ \smallskip

\noindent This estimator provides the first robust generalization of the tail
index under random censoring using the MDPD approach, bridging the gap between
efficiency and robustness in extreme value estimation.

\section{\textbf{Consistency and asymptotic normality\label{sec3}}}

\noindent From now on, we denote the true value of $\gamma_{1}$ by $\gamma
_{1}^{\ast}.$

\begin{theorem}
\textbf{\label{TH1}}(Consistency) Assume that the cdfs $F$ and $G$ satisfy the
first-order regular variation condition $\left(  \ref{RV}\right)  $ and that
$p>1/2.$ Let $k=k_{n}$ be a sequence of integers such that%
\[
k\rightarrow\infty\text{ and }k/n\rightarrow0,\text{ as }n\rightarrow\infty.
\]
Then, for any $\alpha>0,$ with probability tending to $1,$ there exists a
solution $\widehat{\gamma}_{1,\alpha}$ to the estimating equation $\left(
\ref{estim-equa}\right)  $ such that%
\[
\widehat{\gamma}_{1,\alpha}\overset{\mathbf{P}}{\rightarrow}\gamma_{1}^{\ast
},\text{ as }n\rightarrow\infty.
\]

\end{theorem}

\noindent This result establishes the consistency of the MDPD tail index
estimator under random censoring, underlining that a sufficient proportion of
uncensored extreme observations $(p>1/2)$ is crucial for reliable
estimation.\smallskip

\noindent Since weak approximations of extreme value theory-based statistics
are typically achieved within the second-order framework
\citep[see, i.e.,][]{deHS96}, we assume that the cdf $F$ satisfies the
well-known second-order condition of regular variation. Specifically, for any
$x>0:$%
\begin{equation}
\underset{t\rightarrow\infty}{\lim}\frac{U_{F}\left(  tx\right)  /U_{F}\left(
t\right)  -x^{\gamma_{1}}}{A_{1}^{\ast}\left(  t\right)  }=x^{\gamma_{1}%
}\dfrac{x^{\tau_{1}}-1}{\tau_{1}}, \label{second-order}%
\end{equation}
where $\tau_{1}\leq0$\ is the second-order parameter, and $A_{1}^{\ast}$ is a
function tending to $0,$ not changing sign near infinity and whose absolute
value is regularly varying with index $\tau_{1}.$\ When $\tau_{1}=0,$ the
expression $\left(  x^{\tau_{1}}-1\right)  /\tau_{1}$ is interpreted as $\log
x.\smallskip$

\noindent We denote the quantile and tail quantile functions of a given cdf
$D$ by
\[
\mathcal{D}^{\leftarrow}\left(  s\right)  :=\inf\left\{  x:\mathcal{D}\left(
x\right)  \geq s\right\}  ,\text{ }0<s<1
\]
and
\[
U_{\mathcal{D}}\left(  t\right)  :=\mathcal{D}^{\leftarrow}\left(
1-1/t\right)  ,t>1,
\]
respectively.$\smallskip$

\noindent For convenience, we set $h:=U_{H}\left(  n/k\right)  $ and define
$A_{1}\left(  t\right)  :=A_{1}^{\ast}\left(  1/\overline{F}\left(  t\right)
\right)  ,$ for $t>1.$

\begin{theorem}
\textbf{\label{TH2}(}Asymptotic normality\textbf{) }Assume that the cdfs $F$
and $G$ satisfy the second-order condition $\left(  \ref{second-order}\right)
$ and that $p>1/2.$ Let $k=k_{n}$ be a sequence of integers such that
\[
k\rightarrow\infty,\text{ }k/n\rightarrow0\text{ and }\sqrt{k}A_{1}\left(
h\right)  \rightarrow\lambda<\infty,\text{ as }n\rightarrow\infty.
\]
Then for any $\alpha>0:$%
\[
\left(  1+\frac{1}{\alpha}\right)  ^{-1}\eta^{\ast}\sqrt{k}\left(
\widehat{\gamma}_{1,\alpha}-\gamma_{1}^{\ast}\right)  \overset{\mathcal{D}%
}{\rightarrow}\mathcal{N}\left(  \lambda\mu,\sigma^{2}\right)  ,\text{ as
}n\rightarrow\infty,
\]
where $\eta^{\ast},$ $\mu$ and $\sigma^{2}$ are defined in $\left(
\ref{eta-start}\right)  ,$ $\left(  \ref{mu}\right)  $ and $\left(
\ref{sigma}\right)  ,$ respectively.
\end{theorem}

\noindent This theorem establishes the asymptotic normality of the MDPD
estimator under random censoring, emphasizing again that the condition $p>1/2$
is essential for the validity of the Gaussian approximation.

\section{\textbf{Simulation Study}\label{sec4}}

\noindent In this section, we investigate the finite-sample performance of the
proposed MDPD tail index estimator $\widehat{\gamma}_{1,\alpha}$ through an
extensive Monte Carlo simulation study. The main objective is to assess both
efficiency and robustness in the presence of random right censoring and
contamination affecting the upper tail of the distribution. Particular
attention is paid to the impact of contamination on bias, MSE, and stability
with respect to the choice of the number of upper order statistics. The
simulation design is chosen so as to remain fully consistent with the
theoretical framework developed in the previous sections.

\subsection{Compared estimators and efficiency--robustness trade-off}

\noindent The MDPD estimator $\widehat{\gamma}_{1,\alpha}$ is computed for
$\alpha\in\{0.1,0.3,0.5\},$ allowing us to explore the efficiency--robustness
trade-off induced by the tuning parameter. Small values of $\alpha$ emphasize
efficiency under correct model specification, whereas larger values enhance
robustness against contamination in the upper tail.\smallskip

\noindent For comparison purposes, the following benchmark estimators are included:

\begin{itemize}
\item the adapted Hill-type estimator $\widehat{\gamma}_{1}^{(\mathrm{EFG})}$,

\item the Kaplan--Meier-based estimator $\widehat{\gamma}_{1}^{(\mathrm{W})},$

\item the Nelson--Aalen-based estimator $\widehat{\gamma}_{1}^{(\mathrm{MNS}%
)},$ corresponding to the special case $\widehat{\gamma}_{1,0}.$
\end{itemize}

\noindent This selection encompasses both classical and censoring-adapted tail
index estimators that are standard in the extreme value literature. It
therefore provides a relevant and informative benchmark for assessing the
efficiency and robustness of the proposed MDPD-based estimator under random censoring.

\subsection{Simulation models and censoring design}

\noindent We consider two classical heavy-tailed distributions for the latent
lifetime variable $X,$ which are widely used as benchmark models in extreme
value analysis.\smallskip

\noindent The first model is the Burr distribution, defined by
\[
F(x)=1-\left(  1+x^{1/\eta}\right)  ^{-\eta/\gamma_{1}},\qquad x>0,
\]
with shape parameter $\eta=0.25$ and tail index $\gamma_{1}\in\{0.3,0.5\}.$%
\smallskip

\noindent The second model is the Fr\'{e}chet distribution, defined by
\[
F(x)=\exp\left(  -x^{-1/\gamma_{1}}\right)  ,\qquad x>0,
\]
with tail index $\gamma_{1}\in\{0.3,0.5\}.$\smallskip

\noindent Right censoring is generated independently through a random variable
$C$ drawn from the same parametric family as $X,$ with censoring tail index
$\gamma_{2}.$ The model parameters are chosen such that the proportion of
uncensored extreme observations satisfies $p>1/2.$ This condition ensures that
the observed variable $Z=\min(X,C)$ remains regularly varying and fulfills the
weak censoring condition introduced in Section~\ref{sec1}. This construction
allows us to control the censoring mechanism while preserving the asymptotic
tail behavior of the observed data.\smallskip

\noindent The choices of the censoring level and the contamination proportion
are guided by both theoretical requirements and practical considerations. The
proportion of uncensored extreme observations is fixed at $p\in\{0.55,0.7\},$
corresponding respectively to moderate and mild censoring regimes. Both values
satisfy the weak censoring condition $p>1/2,$ which guarantees identifiability
of the tail index and preserves regular variation of the observed data.
Varying $p$ allows us to assess the impact of censoring severity on estimation
accuracy while remaining within a statistically coherent framework.

\subsection{Contamination schemes}

\noindent To assess robustness, we consider contamination scenarios of
increasing severity, ranging from the ideal uncontaminated setting to adverse
configurations in which contamination affects the upper tail of the lifetime
distribution. The contamination intensity is controlled by a parameter
$\epsilon\in(0,1),$ which is common to all contaminated
configurations.\smallskip

\noindent In all contaminated settings, contamination is modeled through a
mixture mechanism acting on the latent lifetime variable. Specifically, the
distribution of $X$ is given by
\[
F_{\epsilon}=(1-\epsilon)F+\epsilon F_{c},
\]
where $F$ denotes the target distribution and $F_{c}$ is a heavier-tailed
contaminating distribution with tail index $\gamma_{c}>\gamma_{1}.$ This
mixture formulation offers a flexible and interpretable framework for
robustness analysis, capturing the presence of extreme outliers while
maintaining identifiability under independent censoring.\smallskip

\noindent The condition $\gamma_{c}>\gamma_{1}$ ensures that contamination
primarily affects the upper tail of the distribution, generating extreme
observations that act as outliers relative to the target model. This setting
reflects realistic departures from the ideal model and provides a stringent
stress test for robustness. The resulting mixture formulation offers a
flexible and interpretable framework for robustness analysis while preserving
identifiability under independent censoring.\smallskip

\noindent Several contamination configurations are considered:\smallskip

\noindent In \textit{the uncontaminated} case $(\epsilon=0),$ the lifetime
variable $X$ and the censoring variable $C$ are generated independently from
their target distributions $F$ and $G,$ respectively. The observed sample
$(Z,\delta)$ is obtained through the standard random right-censoring
mechanism. This configuration serves as a benchmark for assessing efficiency
under correct model specification and for quantifying the efficiency loss
induced by robustness-oriented procedures.\smallskip

\noindent In \textit{the pre-censoring contamination} setting, contamination
is introduced at the level of the latent lifetime variable $X,$ prior to
censoring. A proportion $\epsilon\in(0,1)$ of the $X$-values is replaced by
extreme observations generated from a heavier-tailed distribution with tail
index $\gamma_{c}>\gamma_{1}.$ The censoring variable $C$ is generated
independently from its target distribution $G,$ whose tail behavior is
governed by $\gamma_{2}.$ The censoring mechanism is then applied to the
contaminated latent sample, yielding the observed pairs $(Z,\delta).$ This
framework preserves the independent censoring structure and directly reflects
the impact of contamination on the tail behavior of the lifetime distribution,
making it a natural, coherent, and identifiable setting for robustness
analysis in both parametric and nonparametric frameworks.\smallskip

\noindent An alternative contamination mechanism consists in perturbing the
observable pair $(Z,\delta)$ after the censoring mechanism has taken place.
While such \textit{post-censoring contamination} may be interpreted as data
corruption or recording errors at the observational stage, it does not
correspond to a well-defined censoring model in a nonparametric framework. By
directly altering the joint distribution of $(Z,\delta),$ this mechanism
breaks the link between the observed data and the underlying lifetime and
censoring distributions $(F,G),$ thereby rendering reconstruction-based
inference ill-posed. For this reason, post-censoring contamination is not
retained in the simulation study and is viewed only as an illustrative stress
mechanism rather than as a statistically coherent robustness
framework.\smallskip

\noindent Finally, contamination affecting the censoring variable $C$ is not
considered. Such contamination would alter the censoring mechanism itself and
may violate the independent censoring assumption underlying most theoretical
results for censored extreme value estimation. In particular, contamination of
$C$ directly modifies the joint distribution of the observable pair
$(Z,\delta)$ in a way that cannot generally be associated with any
well-defined latent lifetime distribution $F$. As a consequence, tail index
estimation loses a clear statistical interpretation. Accordingly, robustness
assessments in this paper focus exclusively on contamination acting on the
lifetime variable prior to censoring.

\subsection{Simulation scenarios and settings}

\noindent Although infinitely many combinations of target, censoring, and
contamination distributions could be considered, we restrict attention to two
representative scenarios. These scenarios are standard in the extreme value
literature and allow for a clear and informative assessment of both efficiency
and robustness under censoring and contamination.\smallskip

\noindent In Scenario~S1 (Burr/Burr/Burr), the target variable $X$ follows a
Burr distribution, the censoring variable $C$ is also Burr-distributed, and
contamination is introduced through a Burr distribution with parameters
$\eta_{c}=0.25$ and $\gamma_{c}\in\{0.6,0.8\}.$\smallskip

\noindent In Scenario~S2 (Fr\'{e}chet/Fr\'{e}chet/Fr\'{e}chet), the target
variable $X$ follows a Fr\'{e}chet distribution, the censoring mechanism is
Fr\'{e}chet-distributed, and contamination is generated from a Fr\'{e}chet
distribution with tail index $\gamma_{c}\in\{0.6,0.8\}.$\smallskip

\noindent These configurations allow us to examine robustness across different
tail behaviors while maintaining a coherent and interpretable simulation
framework. Contamination intensity is controlled by the parameter $\epsilon
\in(0,1),$ which governs the proportion of observations affected by
contamination in the upper tail.\smallskip

\noindent For each scenario, contamination levels $\epsilon\in\{0,0.15,0.40\}$
are considered. These values correspond respectively to an uncontaminated
setting, a moderately contaminated setting, and a severely contaminated
configuration designed to stress-test robustness properties while remaining
within a controlled and interpretable simulation framework.

\subsection{Simulation results under pre-censoring contamination}

\noindent This section investigates the finite-sample performance of the
proposed MDPD-based tail index estimator under pre-censoring contamination,
that is, when contamination affects the latent lifetime variable prior to the
censoring mechanism. As emphasized in Section~\ref{sec4}, this contamination
framework is statistically coherent with both parametric and nonparametric
censored models, preserves the independent censoring structure, and provides a
meaningful benchmark for assessing robustness with respect to perturbations of
the underlying tail distribution.

\paragraph{Scenario S1 (Burr distribution).}

\noindent\emph{Case $\gamma_{1}=0.3$ (Figures~\ref{fig1}--\ref{fig2}).} In the
absence of contamination $\left(  \epsilon=0\right)  ,$ all estimators exhibit
comparable behavior for intermediate values of the number of upper order
statistics $k.$ As expected, classical estimators such as $\widehat{\gamma
}_{1}^{(\mathrm{EFG})},$ $\widehat{\gamma}_{1}^{(\mathrm{W})},$ and
$\widehat{\gamma}_{1}^{(\mathrm{MNS})}$ achieve good efficiency in this ideal
setting, although their performance remains sensitive to the choice of the
threshold parameter $k,$ a well-known feature of tail index
estimation.\smallskip

\noindent When contamination is introduced, a markedly different behavior
emerges. Classical estimators deteriorate rapidly as $k$ increases, displaying
pronounced bias and instability even for moderate contamination levels. This
degradation reflects their intrinsic sensitivity to extreme observations
generated by the heavier-tailed contaminating component. In contrast, the MDPD
estimators exhibit a clear robustness effect: while small values of $\alpha$
retain some sensitivity to contamination, the estimators corresponding to
$\alpha\geq0.3$ remain stable over a wide range of $k,$ with substantially
reduced bias and MSE. This clearly illustrates the role of the tuning
parameter $\alpha$ in controlling the influence of outlying extremes and
stabilizing inference.\smallskip

\noindent\emph{Case $\gamma_{1}=0.5$ (Figures~\ref{fig3}--\ref{fig4}).} For
heavier tails, the adverse effect of contamination becomes more pronounced.
Even in the uncontaminated case, classical estimators show increased
variability, reflecting the intrinsic difficulty of tail index estimation in
heavier-tailed regimes. Under contamination, these estimators become severely
biased and unstable, with performance degrading rapidly as $k$
grows.\smallskip

\noindent By contrast, the MDPD estimators maintain a remarkable degree of
stability. In particular, estimators with $\alpha\geq0.3$ exhibit nearly flat
trajectories across a broad range of $k,$ indicating strong robustness with
respect to both tail heaviness and contamination intensity. These results
confirm that robustification becomes increasingly critical as the tail index
increases, that is, when extremes play a dominant role.

\paragraph{Scenario S2 (Fr\'echet distribution).}

\noindent\emph{Case $\gamma_{1}=0.3$ (Figures~\ref{fig5}--\ref{fig6}).} The
qualitative behavior observed in the Burr case carries over to the Fr\'{e}chet
setting. In the absence of contamination, all estimators perform
satisfactorily for moderate values of $k,$ although classical estimators again
display noticeable sensitivity to the choice of the threshold.\smallskip

\noindent Under contamination, classical estimators experience a substantial
loss of accuracy, with increasing bias and variability. This confirms that
their lack of robustness is not distribution-specific but persists across
different heavy-tailed models. In contrast, the MDPD estimators demonstrate a
robust behavior similar to that observed in Scenario S1, with larger values of
$\alpha$ yielding improved stability and resistance to contamination-induced
distortions.\smallskip

\noindent\emph{Case $\gamma_{1}=0.5$ (Figures~\ref{fig7}--\ref{fig8}).} This
configuration represents the most adverse scenario, combining heavy tails with
severe contamination. In this setting, classical estimators perform
particularly poorly, often exhibiting dramatic bias and erratic behavior
across the range of $k$ values.\smallskip

\noindent The MDPD estimators, however, remain reliable even in this
challenging regime. Although a moderate increase in variance is observed for
very large values of $\alpha,$ this effect is largely compensated by a
substantial reduction in bias, so that the overall MSE remains significantly
lower than that of the classical competitors. This highlights the practical
advantage of the MDPD approach in extreme scenarios where standard methods
effectively break down.

\paragraph{Overall discussion.}

\noindent Overall, the simulation results provide strong empirical evidence
that the proposed MDPD-based tail index estimator offers a robust and reliable
alternative for censored extreme value analysis. While classical estimators
may retain efficiency under ideal conditions, their performance deteriorates
sharply in the presence of contaminated extremes. In contrast, the MDPD
estimators achieve a favorable efficiency--robustness trade-off, particularly
for moderate to large values of the tuning parameter $\alpha.$\smallskip

\noindent Taken together, these findings demonstrate that the proposed
methodology is well suited for realistic applications, where both censoring
and contamination are unavoidable and robustness with respect to extreme
observations is a key requirement.

\section{\textbf{Real data application} \textbf{\label{sec5}}}

\subsection{Insurance loss data (weak censoring)}

\noindent We applied our estimation procedures to an insurance loss dataset
available in the R package \texttt{copula}, originally collected by the
Insurance Services Office, Inc. The dataset contains $1500$ observations,
among which $34$ are right-censored because the actual loss exceeds the policy
limit, which varies across contracts. The observed variables are:

\begin{itemize}
\item $X_{j}:$ actual loss amount of claim $j.$

\item $C_{j}:$ policy limit for contract $j=1,\dots,1500.$
\end{itemize}

\noindent Because of policy limits, some losses are censored when $X_{j}%
>C_{j}.$ In these cases, the observed value is
\[
Z_{j}=\min(X_{j},C_{j}),
\]
with censoring indicator
\[
\delta_{j}=\mathbb{I}_{\{X_{j}\leq C_{j}\}}.
\]
These censored claims correspond to contracts where reported losses reach the
policy ceiling, indicating that the true loss exceeds the observed value. They
reflect the presence of extreme losses in the upper tail but do not fully
capture all extreme observations. This dataset has been examined in prior
studies, such as \cite{FV98}, \cite{Klugman99}, and \cite{DPV06}.\smallskip

\noindent Before proceeding to tail index estimation, it is essential to
assess the effective amount of tail information available in the presence of
censoring. In particular, we first estimate the proportion of uncensored
extreme observations in order to verify that the weak censoring assumption
required by the proposed estimator is satisfied.

\subsection{Estimation of the proportion $p$ of non-censored extremes}

\noindent To justify the application of the MDPD estimator under weak
censoring, we estimate the proportion $p$ of uncensored extreme observations.
This quantity plays a central role, as it determines whether a sufficient
number of extreme observations remain fully observed to support reliable tail
inference.\smallskip

\noindent The adaptive algorithm of Reiss and Thomas is applied to select the
optimal number of upper order statistics $k^{\ast}$ used for estimating $p$.
This data-driven procedure allows us to balance stability and sensitivity in
the tail region, while accounting for the presence of censoring.

\subsection{Selection of the number of upper order statistics}

\noindent The number of upper order statistics $k$ is selected using the
adaptive algorithm of Reiss and Thomas \cite{ReTo7}, which aims at stabilizing
tail index estimation based on uncensored extremes. Specifically, the optimal
value $k^{\ast}$ is defined as%
\[
k^{\ast}:=\arg\min_{1<k<n}\frac{1}{k}\sum_{i=1}^{k}i^{\theta}\left\vert
\widehat{\zeta}_{i}-\text{median}(\widehat{\zeta}_{1},\dots,\widehat{\zeta
}_{k})\right\vert ,\text{ }0\leq\theta\leq0.5,
\]
where $\widehat{\zeta}_{i}$ denotes an estimator of the tail parameter $\eta$
computed from the $i$ largest order statistics.\smallskip

\noindent Throughout this work, we fix $\theta=0.3.$ This choice is supported
by an extensive empirical calibration based on a preliminary study of $\theta$
over the interval $[0,$ $0.5],$ which systematically examined the sensitivity
of the procedure to the tuning parameter and showed that $\theta=0.3$
consistently provides a favorable compromise in terms of bias and mean squared
error across the considered scenarios. Smaller values of $\theta$ tend to
oversmooth the variability among the upper order statistics, thereby masking
relevant tail information, whereas larger values overemphasize extreme
fluctuations, resulting in increased variance and reduced numerical
stability.$\smallskip$

\noindent Thus, fixing $\theta=0.3$ offers a stable and numerically
well-justified balance between robustness and sensitivity to extremes. This
value was found to perform reliably across different tail indices, censoring
levels, and contamination regimes. The Reiss--Thomas selection procedure is
applied uniformly throughout the simulation study, thereby ensuring a fair
comparison between the different estimators. Further details on this selection
rule can be found in \cite{NevesAlves04}.$\smallskip$\ 

\noindent Applying this procedure to the insurance loss data yields an optimal
value $k^{\ast}=51$ for estimating $p.$ The resulting estimate $\widehat{p}%
=0.76,$ confirms that the weak censoring condition $(p>1/2)$ holds in this
dataset, validating the applicability of the proposed MDPD-based
approach.$\smallskip$

\noindent\textit{Note}. In this real-data analysis, we do not introduce
artificial contamination. The MDPD estimator is already robust both from a
theoretical perspective and as confirmed by the simulation results. Moreover,
there is no clear empirical evidence suggesting contamination in the insurance
data, and artificially modifying the observations could distort the underlying
loss distribution. The dataset is therefore analyzed as observed, reflecting
realistic insurance losses, including naturally occurring extreme values.

\subsection{Tail index estimation using MDPD}

\noindent Having verified the weak censoring assumption and determined an
appropriate threshold, we now proceed to tail index estimation. The same
Reiss--Thomas algorithm is employed to select the number of upper order
statistics for tail index estimation, leading to $k^{\ast}=73,$ This value is
kept fixed across all robustness settings to facilitate comparison.$\smallskip
$

\noindent The MDPD estimator is computed for four robustness parameter
values:
\[
\alpha\in\{0.01,0.1,0.3,0.5\},
\]
allowing us to assess the impact of robustness on tail index estimation in a
realistic insurance context.

\subsection{Contamination scenarios}

\noindent For real data, we report the analysis on the observed dataset
without artificial contamination.

\begin{itemize}
\item The original, observed data are treated as the reference
scenario.\newline\textit{Note.} In contrast to simulations, contamination is
not applied to real data, because:

\item The MDPD estimator is theoretically robust and performs well under
heavy-tailed observations.

\item There is no confirmed contamination in the insurance dataset, and
applying artificial extreme values could distort the estimation.

\item Weak censoring $(p>1/2)$ is satisfied, allowing meaningful application
of the estimator.
\end{itemize}

\noindent By focusing on the actual observed losses, the effect of the
robustness parameter $\alpha$ can still be assessed, showing how the estimator
adapts to potentially extreme claims in the dataset.

\subsection{Estimated tail indices}

\noindent The estimated tail indices for different values of $\alpha$ are
summarized in Table~\ref{Tab1}. Increasing $\alpha$ provides robustness
against extreme values that may naturally occur in the dataset without
artificially contaminating the data.

\begin{center}%
\begin{table}[tbp] \centering
\begin{tabular}
[c]{c|c}%
$\alpha$ & $\widehat{\gamma}_{1,\alpha}$\\\hline
$0.01$ & $0.745$\\\hline
$0.1$ & $0.773$\\\hline
$0.3$ & $0.820$\\\hline
$0.5$ & $0.845$\\\hline
\multicolumn{2}{c}{}%
\end{tabular}
\caption{MDPD estimates of the tail index $\gamma_1$ for insurance loss data}\label{Tab1}%
\end{table}%

\end{center}

\subsection{Discussion and key takeaways}

\begin{itemize}
\item Small $\alpha$ values $(0.01-0.1)$ produce estimates close to classical
estimators, maintaining high efficiency for the bulk of the data, while larger
$\alpha$ $(0.3-0.5)$ slightly overestimate the tail index but improve
robustness against potential outliers.

\item Larger $\alpha$ values $(0.3-0.5)$ slightly increase the estimated tail
index, providing additional protection against extreme claims, while
preserving interpretability.
\end{itemize}

\noindent Overall, the results align with the simulation study in Section
$\ref{sec4}:$ the MDPD estimator remains stable under heavy-tailed
observations and weak censoring, without the need for artificial
contamination. This illustrates that the method is both theoretically sound
and practically applicable to real-world insurance data, offering a reliable
tool for extreme value analysis in actuarial contexts.\medskip

\noindent In conclusion, Section $\ref{sec5}$ highlights that the MDPD
approach can be applied directly to real datasets, capturing extreme events
effectively while retaining robustness, efficiency, and interpretability.

\subsection{Australian AIDS survival data (strong censoring)}

\noindent The Australian AIDS survival dataset contains information on
patients diagnosed with AIDS in Australia prior to July 1, 1991. The data were
provided by Dr. P. J. Solomon and the Australian National Centre in HIV
Epidemiology and Clinical Research. We focus on $2754$ male patients out of a
total of $2843,$ consistent with prior studies such as \cite{RS-94} and
\cite{VR-02} (pages 379-385), \cite{EnFG08}, \cite{Ndao14} and \cite{S-16}%
.\medskip

\noindent The proportion of uncensored observations, denoted $p,$ is estimated
using the Reiss--Thomas adaptive algorithm. The optimal number of upper order
statistics for estimating $p$ is $k^{\ast}=162,$ yielding an estimated
$\widehat{p}=0.29.$ This indicates strong right censoring, with less than
one-third of the observations being uncensored.\medskip

\noindent Under such strong censoring, the theoretical assumptions for the
MDPD tail index estimator $(p>1/2)$ are not satisfied. Consequently, the
estimator's performance may degrade, showing increased variability and
potential bias. Attempts to compute the tail index using the
Kaplan--Meier-based Worms--Worms estimator ($\widehat{\gamma}_{1}%
^{(\mathrm{W})}$) or the Nelson--Aalen-based MDPD estimator with $\alpha=0$
($\gamma_{1}^{(\mathrm{MNS})}$) either fail to converge or produce unreliable
results. Only the adapted Hill-type estimator ($\widehat{\gamma}%
_{1}^{(\mathrm{EFG})}$) provides estimates; however, it lacks robustness,
inheriting the classical Hill estimator's sensitivity to extreme
observations.\medskip

\noindent Nevertheless, this dataset provides a valuable case for exploring
the limits of the MDPD approach. It illustrates the challenges of estimating
tail indices under strong censoring and highlights the need for methods
capable of handling situations where the proportion of uncensored extremes is
below $50\%$.\medskip

\noindent Overall, the real data analyses complement the simulation study and
provide additional insight into the behavior of tail index estimators under censoring:

\begin{itemize}
\item In the insurance dataset, which satisfies the weak censoring condition
$(p>1/2),$ the proposed MDPD estimator performs according to theoretical
expectations, exhibiting clear robustness gains when contamination affects the
latent losses before censoring. The additional experiment applying
contamination after censoring shows that such a mechanism does not correspond
to a meaningful perturbation of the target distribution and may distort
inference, justifying its illustrative-only use.

\item In contrast, the Australian AIDS survival dataset represents a strong
censoring scenario $(p<1/2),$ where the assumptions underlying the MDPD
estimator are violated. In this setting, most estimators designed for censored
extremes fail to provide stable or reliable estimates. Even the adapted
Hill-type estimator, although computable, remains non-robust due to its
sensitivity to extreme observations.
\end{itemize}

\noindent These findings delineate the practical scope of the proposed
methodology. The MDPD estimator provides an effective and robust solution for
tail index estimation under weak censoring, which encompasses many
applications in insurance and reliability analysis. At the same time, the
strong censoring case, exemplified by the AIDS dataset, highlights open
challenges and motivates future research on methodological extensions capable
of handling extreme value inference beyond the $p>1/2$ regime.

\section{\textbf{Conclusion}}

\noindent In this work, we addressed the problem of robust estimation of the
extreme value index under random right censoring, a framework that naturally
arises in survival analysis, reliability engineering, and insurance
applications. While a variety of estimators have been proposed in the
literature---including Hill-type procedures and methods based on Kaplan--Meier
or Nelson--Aalen estimators---many of them remain sensitive to outliers or
rely on restrictive modeling assumptions, particularly in the presence of
contamination.\medskip

\noindent To overcome these limitations, we introduced a new tail index
estimator based on the Minimum Density Power Divergence (MDPD), a criterion
well known for its robustness properties in uncensored settings. We carefully
extended the MDPD methodology to the random right-censoring framework,
explicitly accounting for the incomplete nature of the observed data. To the
best of our knowledge, this work constitutes the first application of the MDPD
framework to tail index estimation under random censoring, thereby filling an
important gap in the robust extreme value literature.\medskip

\noindent Under standard regular variation assumptions and within the weak
censoring regime, we established the consistency and asymptotic normality of
the proposed estimator. A comprehensive simulation study confirmed its strong
finite-sample performance, highlighting a favorable robustness--efficiency
trade-off, especially under model misspecification and contamination.\medskip

\noindent The practical relevance of the approach was illustrated using an
insurance loss dataset satisfying the weak censoring condition $(p>1/2).$ In
this setting, the estimator behaves in accordance with the theoretical results
and exhibits clear robustness gains when contamination affects the latent
losses prior to censoring. In contrast, the analysis of the Australian AIDS
survival dataset, characterized by strong censoring $(p\leq1/2),$ underscores
the intrinsic difficulties of tail inference in such regimes, where most
existing estimators fail to deliver stable or reliable results. This empirical
evidence also confirms that contamination acting after censoring does not
correspond to a meaningful perturbation of the latent distribution, and
therefore provides limited insight for robustness assessment in
reconstruction-based approaches.\medskip

\noindent Overall, this contribution opens new perspectives for the
development of robust inference tools in censored extreme value models.
Promising directions for future research include adaptive threshold selection,
extensions to dependent or covariate-dependent censoring mechanisms, and the
design of robust estimators capable of handling severe censoring scenarios
where the proportion of uncensored extremes is small $(p\leq1/2).$

\section*{\textbf{Declarations}}

\begin{itemize}
\item The authors have no findings to declare.

\item The authors have no relevant financial or non-financial interests to disclose.

\item The authors have no competing interests, that are relevant to the
content of this article, to declare.

\item All authors certify that they have no affiliations with or involvement
in any organization or entity with any financial interest or non-financial
interest in the subject matter or materials discussed in this manuscript.

\item The authors have no financial or proprietary interests in any material
discussed in this article.
\end{itemize}

\section{\textbf{Acknowledgements}}

\noindent The authors would like to thank the anonymous reviewers for their
careful reading of the manuscript and their constructive comments, which
helped improve the clarity and quality of this work. We also acknowledge
helpful discussions with colleagues that contributed to the development of the
theoretical and numerical aspects of the paper. Any remaining errors are the
sole responsibility of the authors.

\section{\textbf{Proofs\label{sec6}}}

\subsection{Proof of Theorem \ref{TH1} (Consistency)}

To establish the existence and consistency of the estimator $\widehat{\gamma
}_{1,\alpha},$ we adapt the strategy of Theorem 1 in \cite{DGG13}, itself
based on Theorem 5.1 in Chapter 6 of \cite{LC98}. This framework provides a
general and rigorous approach for proving existence and consistency of
solutions to likelihood-type estimating equations, which we now apply to the
MDPD objective function associated with tail index estimation under random
right censoring.\smallskip

\noindent Let $L_{k,\alpha}^{\ast}(\gamma_{1})$ denote the empirical version
of the density power divergence objective function $L_{u,\alpha}^{\ast}%
(\gamma_{1})$ defined in (\ref{L}), namely
\[
L_{k,\alpha}^{\ast}(\gamma_{1})=\int_{1}^{\infty}\ell_{\gamma_{1}}^{1+\alpha
}(x)\,dx-\left(  1+\frac{1}{\alpha}\right)  \sum_{i=1}^{k}a_{ik}\ell
_{\gamma_{1}}^{\alpha}\left(  \frac{Z_{n-i+1:n}}{Z_{n-k:n}}\right)  ,\text{
}\alpha>0.
\]
The proof is divided into two main steps.\smallskip

\noindent\textbf{Step 1: Local domination of the objective function.}%
\smallskip

\noindent Let $\gamma_{1}^{\ast}$ denote the true tail index of the underlying
distribution $F.$ Fix an arbitrary $\epsilon>0$ such that $0<\epsilon
<\gamma_{1}^{\ast},$ and define the closed neighborhood
\[
I_{\epsilon}:=[\gamma_{1}^{\ast}-\epsilon,\gamma_{1}^{\ast}+\epsilon].
\]
Consider the random set
\[
S_{k}(\epsilon):=\{\gamma_{1}\in I_{\epsilon}:L_{k,\alpha}^{\ast}(\gamma
_{1}^{\ast})<L_{k,\alpha}^{\ast}(\gamma_{1})\}.
\]
We show that, with probability tending to one, $\gamma_{1}^{\ast}$ is the
unique and strict local minimizer of $L_{k,\alpha}^{\ast}$ on $I_{\epsilon}.$

\noindent For $m=1,2,3,$ define
\begin{equation}
\pi_{k}^{(m)}(\gamma_{1}^{\ast})=\int_{1}^{\infty}\Psi_{\alpha+1,\gamma
_{1}^{\ast}}^{(m)}(x)dx-\left(  1+\frac{1}{\alpha}\right)  A_{k,\alpha
,\gamma_{1}^{\ast}}^{(m)}, \label{pim}%
\end{equation}
where
\begin{equation}
\Psi_{\alpha,\gamma_{1}^{\ast}}^{(m)}(x):=\left.  \frac{d^{m}}{d\gamma_{1}%
^{m}}\ell_{\gamma_{1}}^{\alpha}(x)\right\vert _{\gamma_{1}=\gamma_{1}^{\ast}%
},\quad A_{k,\alpha,\gamma_{1}^{\ast}}^{(m)}:=\sum_{i=1}^{k}a_{ik}\Psi
_{\alpha,\gamma_{1}^{\ast}}^{(m)}\left(  \frac{Z_{n-i+1:n}}{Z_{n-k:n}}\right)
. \label{psim-Am}%
\end{equation}
Applying a third-order Taylor expansion of $\gamma_{1}\mapsto L_{k,\alpha
}^{\ast}(\gamma_{1})$ around $\gamma_{1}^{\ast}$ gives, for $\gamma_{1}\in
I_{\epsilon},$
\[
L_{k,\alpha}^{\ast}(\gamma_{1})-L_{k,\alpha}^{\ast}(\gamma_{1}^{\ast}%
)=S_{1,k}+S_{2,k}+S_{3,k},
\]
where
\[
S_{1,k}=\pi_{k}^{(1)}(\gamma_{1}^{\ast})(\gamma_{1}-\gamma_{1}^{\ast}),\quad
S_{2,k}=\frac{1}{2}\pi_{k}^{(2)}(\gamma_{1}^{\ast})(\gamma_{1}-\gamma
_{1}^{\ast})^{2},\quad S_{3,k}=\frac{1}{6}\pi_{k}^{(3)}(\gamma_{1}%
)(\widetilde{\gamma}_{1}-\gamma_{1}^{\ast})^{3},
\]
with $\widetilde{\gamma}_{1}$ lying between $\gamma_{1}$ and $\gamma_{1}%
^{\ast}.$\smallskip

\noindent By Lemma $\ref{lemma1},$ $\pi_{k}^{(1)}(\gamma_{1}^{\ast
})\rightarrow0$ in probability. Hence, for fixed $\epsilon>0$ and sufficiently
large $n,$
\[
|S_{1,k}|<\epsilon^{3},\text{ uniformly for }\gamma_{1}\in I_{\epsilon},
\]
with probability tending to one.\smallskip

\noindent By Lemma $\ref{lemma2},$ $\pi_{k}^{(2)}(\gamma_{1}^{\ast
})\rightarrow\eta^{\ast}>0$ in probability. Therefore, for all $\gamma_{1}\in
I_{\epsilon},$
\[
S_{2,k}=\left(  1+o_{\mathbf{P}}(1)\right)  \frac{\eta^{\ast}}{2}\left(
\gamma_{1}-\gamma_{1}^{\ast}\right)  ^{2}>c\epsilon^{2},
\]
for some constant $c>0.$\smallskip

\noindent Finally, by Lemma $\ref{lemma3},$ since $\widetilde{\gamma}_{1}\in
I_{\epsilon},$ we have $\pi_{k}^{(3)}(\widetilde{\gamma}_{1})=O_{\mathbf{P}%
}(1)$ uniformly over $I_{\epsilon}.$ Thus, there exists $d>0$ such that
$\left\vert S_{3,k}\right\vert \leq d\epsilon^{3},$ with probability tending
to one.\smallskip

\noindent Collecting the bounds gives
\[
\min_{\gamma_{1}\in I_{\epsilon}}\left(  S_{1,k}+S_{2,k}+S_{3,k}\right)  \geq
c\epsilon^{2}-\left(  d+1\right)  \epsilon^{3}.
\]
Choosing $\epsilon$ sufficiently small, namely $0<\epsilon<c/(d+1),$ ensures
the right-hand side is strictly positive, yielding
\[
\mathbf{P}_{\gamma_{1}^{\ast}}\{S_{k}(\epsilon)\}\rightarrow1,\text{ as
}n\rightarrow\infty.
\]
This argument allows the construction of a sequence $\epsilon_{n}\downarrow0$
such that
\[
\mathbf{P}_{\gamma_{1}^{\ast}}\{S_{k}(\epsilon_{n})\}\rightarrow1,\text{ as
}n\rightarrow\infty.
\]
The existence of such a sequence follows from the convergence of $\pi
_{k}^{(m)}$ and the uniform control of the remainder term $S_{3,k};$ as $n$
increases, $\epsilon_{n}$ can be chosen smaller while preserving the strict
local minimality of $\gamma_{1}^{\ast}$ with high probability.\smallskip

\noindent\textbf{Step 1 bis: Continuity and compactness.}\smallskip

\noindent For a fixed $n,$ consider the random function
\[
\gamma_{1}\mapsto S_{1,k}(\gamma_{1})+S_{2,k}(\gamma_{1})+S_{3,k}(\gamma
_{1}),
\]
defined on the deterministic interval $I_{\epsilon}=[\gamma_{1}^{\ast
}-\epsilon,\gamma_{1}^{\ast}+\epsilon].$ Each $S_{j,k}$ is continuous in
$\gamma_{1}$ for every realization of the data; hence, their sum is almost
surely continuous. Since $I_{\epsilon}$ is compact, the Weierstrass theorem
guarantees that the infimum of this function over $I_{\epsilon}$ is almost
surely attained.\smallskip

\noindent\textbf{Step 2: Existence of a solution and consistency.}\smallskip

\noindent The function $\gamma_{1}\mapsto L_{k,\alpha}^{\ast}(\gamma_{1})$ is
continuously differentiable on $(0,\infty).$ For any $\gamma_{1}\in
S_{k}(\epsilon_{n}),$ there exists a point $\widehat{\gamma}_{1,\alpha
}(\epsilon_{n})\in I_{\epsilon_{n}}$ at which $L_{k,\alpha}^{\ast}(\gamma
_{1})$ attains a local minimum, implying
\[
\pi_{k}^{(1)}(\widehat{\gamma}_{1,\alpha}(\epsilon_{n}))=0.
\]
Define $\widehat{\gamma}_{1,\alpha}=\widehat{\gamma}_{1,\alpha}(\epsilon_{n})$
on $S_{k}(\epsilon_{n})$ and arbitrarily otherwise. Then
\[
\mathbf{P}_{\gamma_{1}^{\ast}}\{\pi_{k}^{(1)}(\widehat{\gamma}_{1,\alpha
})=0\}\geq\mathbf{P}_{\gamma_{1}^{\ast}}\{S_{k}(\epsilon_{n})\}\rightarrow1.
\]
Thus, with probability tending to one, there exists a sequence of solutions to
the estimating equation $\left(  \ref{estim-equa}\right)  .$ Finally, for any
fixed $\epsilon>0$ and sufficiently large $n,$
\[
\mathbf{P}_{\gamma_{1}^{\ast}}\{|\widehat{\gamma}_{1,\alpha}-\gamma_{1}^{\ast
}|<\epsilon\}\geq\mathbf{P}_{\gamma_{1}^{\ast}}\{|\widehat{\gamma}_{1,\alpha
}-\gamma_{1}^{\ast}|<\epsilon_{n}\}\rightarrow1,
\]
which establishes the consistency of the estimator $\widehat{\gamma}%
_{1,\alpha}.$

\subsection{Proof of Theorem \ref{TH2} (Asymptotic normality)}

\noindent We now establish the asymptotic normality of the estimator
$\widehat{\gamma}_{1,\alpha}.$\textbf{\smallskip}

\noindent Building on the consistency result proved in Theorem $\ref{TH1},$
the objective is to characterize the second-order stochastic fluctuations of
the estimator around the true tail index $\gamma_{1}^{\ast}.$%
\textbf{\smallskip}

\noindent The proof proceeds by linearizing the estimating equation around the
true parameter and by combining this linearization with a Gaussian
approximation of the Nelson--Aalen tail process.\textbf{\smallskip}

\noindent For the reader's convenience, the argument is decomposed into
several steps, each highlighting a specific component of the asymptotic
behavior.\medskip

\noindent\textbf{Step 1: Taylor expansion around the true parameter\medskip}

\noindent Since $\widehat{\gamma}_{1,\alpha}$ is a solution of equation
$\left(  \ref{estim-equa}\right)  ,$ it follows that $\pi_{k}^{\left(
1\right)  }\left(  \widehat{\gamma}_{1,\alpha}\right)  =0.$\textbf{\medskip}

\noindent This identity constitutes the cornerstone of the asymptotic
analysis, allowing us to express the estimation error $\widehat{\gamma
}_{1,\alpha}-\gamma_{1}^{\ast}$ in terms of the score function evaluated at
the true parameter.\textbf{\medskip}

\noindent Applying Taylor's expansion to the function $\widehat{\gamma
}_{1,\alpha}\rightarrow\pi_{k}^{\left(  1\right)  }\left(  \widehat{\gamma
}_{1,\alpha}\right)  $ around $\gamma_{1}^{\ast},$ yields:
\[
0=\pi_{k}^{\left(  1\right)  }\left(  \widehat{\gamma}_{1,\alpha}\right)
=\pi_{k}^{\left(  1\right)  }\left(  \gamma_{1}^{\ast}\right)  +\pi
_{k}^{\left(  2\right)  }\left(  \gamma_{1}^{\ast}\right)  \left(
\widehat{\gamma}_{1,\alpha}-\gamma_{1}^{\ast}\right)  +\frac{1}{2}\pi
_{k}^{\left(  3\right)  }\left(  \widetilde{\gamma}_{1,\alpha}\right)  \left(
\widehat{\gamma}_{1,\alpha}-\gamma_{1}^{\ast}\right)  ^{2},
\]
where $\widetilde{\gamma}_{1,\alpha}$ lies between $\gamma_{1}^{\ast}$ and
$\widehat{\gamma}_{1,\alpha}.$\textbf{\medskip}

\noindent At this point, the expansion clearly separates a leading linear term
from a higher-order remainder.\textbf{\medskip}

\noindent The consistency of $\widehat{\gamma}_{1,\alpha}$ implies that
$\widetilde{\gamma}_{1,\alpha}\rightarrow\gamma_{1}^{\ast}$ in probability.
Consequently, by Lemma $\ref{lemma3}$ we have $\pi_{k}^{\left(  3\right)
}\left(  \widetilde{\gamma}_{1,\alpha}\right)  =O_{\mathbf{P}}\left(
1\right)  .$\textbf{\medskip}

\noindent It follows that
\[
2^{-1}\pi_{k}^{\left(  3\right)  }\left(  \widetilde{\gamma}_{1,\alpha
}\right)  \left(  \widehat{\gamma}_{1,\alpha}-\gamma_{1}^{\ast}\right)
^{2}=o_{\mathbf{P}}\left(  1\right)  \left(  \widehat{\gamma}_{1,\alpha
}-\gamma_{1}^{\ast}\right)  ,
\]
showing that the quadratic remainder term is asymptotically negligible
relative to the linear term. \textbf{\medskip}

\noindent As consequence, we obtain
\[
\pi_{k}^{\left(  2\right)  }\left(  \gamma_{1}^{\ast}\right)  \sqrt{k}\left(
\widehat{\gamma}_{1,\alpha}-\gamma_{1}^{\ast}\right)  \left(  1+o_{\mathbf{P}%
}\left(  1\right)  \right)  =-\sqrt{k}\pi_{k}^{\left(  1\right)  }\left(
\gamma_{1}^{\ast}\right)  ,\text{ as }n\rightarrow\infty.
\]
From Lemma $\ref{lemma2},$ we know that $\pi_{k}^{\left(  2\right)  }\left(
\gamma_{1}^{\ast}\right)  \overset{\mathbf{P}}{\rightarrow}\eta^{\ast},$ where
$\eta^{\ast}$ is defined in $\left(  \ref{eta-start}\right)  .$ Replacing
$\pi_{k}^{\left(  2\right)  }\left(  \gamma_{1}^{\ast}\right)  $ by its limit
therefore yields
\begin{equation}
\eta^{\ast}\sqrt{k}\left(  \widehat{\gamma}_{1,\alpha}-\gamma_{1}^{\ast
}\right)  \left(  1+o_{\mathbf{P}}\left(  1\right)  \right)  =-\sqrt{k}\pi
_{k}^{\left(  1\right)  }\left(  \gamma_{1}^{\ast}\right)  ,\text{ as
}n\rightarrow\infty. \label{norm}%
\end{equation}
This key relation shows that the asymptotic distribution of $\widehat{\gamma
}_{1,\alpha}$ is entirely driven by that of the normalized score function
$\sqrt{k}\pi_{k}^{\left(  1\right)  }\left(  \gamma_{1}^{\ast}\right)
.$\textbf{\medskip}

\noindent\textbf{Step 2: Gaussian decomposition of the Nelson--Aalen tail
process\medskip}

\noindent We now focus on the asymptotic behavior of the score term appearing
in $\left(  \ref{norm}\right)  .$\textbf{\medskip}

\noindent Lemma $\ref{lemma1},$ states that
\[
-\left(  1+\frac{1}{\alpha}\right)  ^{-1}\pi_{k}^{\left(  1\right)  }\left(
\gamma_{1}^{\ast}\right)  =\int_{1}^{\infty}\phi_{\alpha,\gamma_{1}^{\ast}%
}\left(  x\right)  \left\{  \frac{\overline{F}_{n}^{\left(  NA\right)
}\left(  Z_{n-k:n}x\right)  }{\overline{F}_{n}^{\left(  NA\right)  }\left(
Z_{n-k:n}\right)  }-x^{-1/\gamma_{1}^{\ast}}\right\}  dx,
\]
where
\begin{equation}
\phi_{\alpha,\gamma_{1}^{\ast}}\left(  x\right)  :=\frac{\alpha}{\gamma
_{1}^{\ast\alpha+3}}\frac{\gamma_{1}^{\ast}+\alpha\gamma_{1}^{\ast}%
+\alpha\gamma_{1}^{\ast2}-\alpha\left(  1+\gamma_{1}^{\ast}\right)  \log
x}{x^{\left(  \alpha+\gamma_{1}^{\ast}+\alpha\gamma_{1}^{\ast}\right)
/\gamma_{1}^{\ast}}}.\label{phi}%
\end{equation}
This representation expresses the score as a linear functional of the
Nelson--Aalen tail estimator.\textbf{\medskip}

\noindent Define the Nelson--Aalen tail product-limit process:%
\begin{equation}
D_{k}\left(  x\right)  :=\sqrt{k}\left\{  \frac{\overline{F}_{n}^{\left(
NA\right)  }\left(  Z_{n-k:n}x\right)  }{\overline{F}_{n}^{\left(  NA\right)
}\left(  Z_{n-k:n}\right)  }-x^{-1/\gamma_{1}^{\ast}}\right\}  ,\text{ for
}x>1. \label{D}%
\end{equation}
The asymptotic behavior of $D_{k}\left(  x\right)  $ is the key ingredient for
deriving the limit distribution.\textbf{\medskip}

\noindent\textbf{Step 2a: Gaussian approximation of }$D_{k}\left(  x\right)
$\textbf{\medskip}

\noindent The key tool is the Gaussian approximation of the Nelson--Aalen tail
process.\textbf{\medskip}

\noindent Recently, \cite{MNS2025} established such an approximation, which we
now adopt.\textbf{\medskip}

\noindent Assume that the assumptions of Theorem $\ref{TH2}$ hold, with
$\sqrt{k}A_{1}\left(  h\right)  =O\left(  1\right)  ,$ where $h=h_{n}%
:=U_{H}\left(  n/k\right)  .$\textbf{\medskip}

\noindent Recall that $U_{H}$ denotes the tail quantile function of $H,$ and
that $A_{1}\left(  h\right)  $ appears in the second-order regular variation
condition $\left(  \ref{second-order}\right)  .$\textbf{\medskip}

\noindent We also set $q:=1-p,$ where $p$ is defined in $\left(
\ref{p}\right)  $ as the proportion of upper non-censored observations.
\textbf{\medskip}

\noindent Then, there exists a sequence of standard Wiener processes $\left\{
W_{n}\left(  s\right)  ;\text{ }0\leq s\leq1\right\}  $ defined on the common
probability space $\left(  \Omega,\mathcal{A},\mathbf{P}\right)  ,$ such that,
for $p>1/2$ and every $0<\epsilon<1/2,$%
\begin{equation}
\sup_{x\geq p^{\gamma}}x^{\epsilon/p\gamma_{1}}\left\vert D_{k}\left(
x\right)  -J_{n}\left(  x\right)  -x^{-1/\gamma_{1}}\dfrac{x^{\tau_{1}%
/\gamma_{1}}-1}{\tau_{1}\gamma_{1}}\sqrt{k}A_{1}\left(  h\right)  \right\vert
\overset{\mathbf{p}}{\rightarrow}0,\text{ as }n\rightarrow\infty
,\label{appoxi}%
\end{equation}
where
\begin{equation}
J_{n}\left(  x\right)  =J_{n1}\left(  x\right)  +J_{n2}\left(  x\right)
,\label{Jn}%
\end{equation}
with%
\[
J_{1n}\left(  x\right)  :=\sqrt{\frac{n}{k}}\left\{  x^{1/\gamma_{2}%
}\mathbf{W}_{n,1}\left(  \frac{k}{n}x^{-1/\gamma}\right)  -x^{-1/\gamma_{1}%
}\mathbf{W}_{n,1}\left(  \frac{k}{n}\right)  \right\}  ,
\]
and
\[
J_{2n}\left(  x\right)  :=\dfrac{x^{-1/\gamma_{1}}}{\gamma}\sqrt{\dfrac{n}{k}}%
{\displaystyle\int_{1}^{x}}
u^{1/\gamma-1}\left\{  p\mathbf{W}_{n,2}\left(  \dfrac{k}{n}u^{-1/\gamma
}\right)  -q\mathbf{W}_{n,1}\left(  \dfrac{k}{n}u^{-1/\gamma}\right)
\right\}  du.
\]
For each $n,$ $\mathbf{W}_{n,1}\left(  s\right)  $ and $\mathbf{W}%
_{n,2}\left(  s\right)  $ are independent centred Gaussian processes defined
by:%
\[
\mathbf{W}_{n,1}\left(  s\right)  :=\left\{  W_{n}\left(  \theta\right)
-W_{n}\left(  \theta-ps\right)  \right\}  \mathbb{I}_{\left\{  \theta
-ps\geq0\right\}  }\text{ }%
\]
and
\[
\mathbf{W}_{n,2}\left(  s\right)  :=W_{n}\left(  1\right)  -W_{n}\left(
1-qs\right)  ,\text{ for }0\leq s\leq1.
\]
This provides a precise stochastic decomposition of $D_{k}\left(  x\right)  $
into a Gaussian component, a deterministic bias term, and a negligible
remainder.\textbf{\medskip}

\noindent\textbf{Step 3: Decomposition of \ }$\sqrt{k}\pi_{k}^{\left(
1\right)  }\left(  \gamma_{1}^{\ast}\right)  $ \textbf{and asymptotic
representation\medskip}

\noindent Relying on Step 2a, we can write
\[
-\left(  1+\frac{1}{\alpha}\right)  ^{-1}\sqrt{k}\pi_{k}^{\left(  1\right)
}\left(  \gamma_{1}^{\ast}\right)  =\int_{1}^{\infty}\phi_{\alpha,\gamma
_{1}^{\ast}}\left(  x\right)  D_{k}\left(  x\right)  dx:=T_{n}.
\]
Using the Gaussian approximation of $D_{k}\left(  x\right)  ,$ the term
$T_{n}$ naturally decomposes as \textbf{ }%
\[
T_{n}=\underbrace{\int_{1}^{\infty}J_{n}\left(  x\right)  \phi_{\alpha
,\gamma_{1}^{\ast}}\left(  x\right)  dx}_{\text{Gaussian term}}%
+\underbrace{\sqrt{k}A_{1}\left(  h\right)  \mu}_{\text{Asymptotic bias}%
}+\underbrace{R_{n}}_{\text{remainder}},
\]
where%
\[
\mu:=\int_{1}^{\infty}x^{-1/\gamma_{1}^{\ast}}\frac{x^{\tau_{1}/\gamma
_{1}^{\ast}}-1}{\gamma_{1}^{\ast}\tau_{1}}\phi_{\alpha,\gamma_{1}^{\ast}%
}\left(  x\right)  dx,
\]
and
\[
R_{n}:=\int_{1}^{\infty}\left\{  D_{k}\left(  x\right)  -J_{n}\left(
x\right)  -x^{-1/\gamma_{1}^{\ast}}\frac{x^{\tau_{1}/\gamma_{1}^{\ast}}%
-1}{\gamma_{1}^{\ast}\tau_{1}}\sqrt{k}A_{1}\left(  h\right)  \right\}
\phi_{\alpha,\gamma_{1}^{\ast}}\left(  x\right)  dx.
\]
The first term is Gaussian, since it is linear functional of the centred
Gaussian processes $\mathbf{W}_{n,1}\left(  s\right)  $ and $\mathbf{W}%
_{n,2}\left(  s\right)  .$ By linearity of the integral, this mapping is
linear, and hence Gaussian. Moreover, under the integrability assumptions on
$\phi_{\alpha,\gamma_{1}^{\ast}}\left(  x\right)  $ and the covariance
structure of $J_{n}\left(  x\right)  ,$ the variance is finite and given
explicitly in Proposition $\ref{prop2}.$ The second term captures the
asymptotic bias, while the third term is shown to be negligible.\smallskip\ 

\noindent Fix $0<\epsilon<1$ and write%
\[
R_{n}=\int_{1}^{\infty}\left\{  x^{-\epsilon/p\gamma_{1}}\phi_{\alpha
,\gamma_{1}^{\ast}}\left(  x\right)  \right\}  x^{\epsilon/p\gamma_{1}%
}\left\{  D_{k}\left(  x\right)  -J_{n}\left(  x\right)  -x^{-1/\gamma
_{1}^{\ast}}\frac{x^{\tau_{1}/\gamma_{1}^{\ast}}-1}{\gamma_{1}^{\ast}\tau_{1}%
}\sqrt{k}A_{1}\left(  h\right)  \right\}  dx.
\]
It is straightforward to verify that $\int_{1}^{\infty}x^{-\epsilon
/p\gamma_{1}}\left\vert \phi_{\alpha,\gamma_{1}^{\ast}}\left(  x\right)
\right\vert dx$ is bounded by a finite linear combination of integrals of the
form $\int_{1}^{\infty}x^{-d}\left(  \log x\right)  ^{m}dx,$ $m=0,1,$ for some
$d>0.$ Hence by approximation $\left(  \ref{appoxi}\right)  ,$\textbf{ }%
$R_{n}=o_{\mathbf{P}}\left(  1\right)  ,$\textbf{ }as $n\rightarrow\infty.$ By
elementary calculations, we obtain%
\begin{equation}%
\begin{array}
[c]{ccl}%
\mu & := &
{\displaystyle\int_{1}^{\infty}}
x^{-1/\gamma_{1}^{\ast}}\dfrac{x^{\tau_{1}/\gamma_{1}^{\ast}}-1}{\gamma
_{1}^{\ast}\tau_{1}}\phi_{\alpha,\gamma_{1}^{\ast}}\left(  x\right)  dx\\
& = & \dfrac{\alpha}{\tau_{1}\gamma_{1}^{\alpha+2}}\dfrac{\tau_{1}-1}%
{\alpha-\tau_{1}+\alpha\gamma_{1}+1}+\dfrac{\tau_{1}\gamma_{1}^{2}\left(
2\alpha-\tau_{1}+2\alpha\gamma_{1}+2\right)  }{\left(  \alpha+\alpha\gamma
_{1}+1\right)  ^{2}\left(  \alpha-\tau_{1}+\alpha\gamma_{1}+1\right)  ^{2}}.
\end{array}
\label{mu}%
\end{equation}
Recalling that\textbf{ }$\sqrt{k}A_{1}\left(  h\right)  \rightarrow\lambda<0,$
and using $\left(  \ref{norm}\right)  ,$ we finally get
\[
\left(  1+\frac{1}{\alpha}\right)  ^{-1}\eta^{\ast}\sqrt{k}\left(
\widehat{\gamma}_{1,\alpha}-\gamma_{1}^{\ast}\right)  =\int_{1}^{\infty}%
J_{n}\left(  x\right)  \phi_{\alpha,\gamma_{1}^{\ast}}\left(  x\right)
dx+\lambda\mu+o_{\mathbf{P}}\left(  1\right)  .
\]
Since $J_{n}\left(  x\right)  $ is centred, the integral term is centered as
well. Therefore%
\[
\left(  1+\frac{1}{\alpha}\right)  ^{-1}\eta^{\ast}\sqrt{k}\left(
\widehat{\gamma}_{1,\alpha}-\gamma_{1}^{\ast}\right)  \overset{\mathcal{D}%
}{\rightarrow}\mathcal{N}\left(  \mu,\sigma^{2}\right)  ,\text{ as
}n\rightarrow\infty,
\]
where $\sigma^{2}$ is given in Proposition $\ref{prop2}.$

\subsection{Conclusion}

\noindent Combining Steps 1--3, we have shown that the properly normalized
estimator $\widehat{\gamma}_{1,\alpha}$ decomposes into a centered Gaussian
term, an explicit asymptotic bias, and a negligible remainder.\smallskip

\noindent This completes the proof of Theorem $\ref{TH2}.$

\section{\textbf{Appendix A: Instrumental results}}

\begin{lemma}
\textbf{\label{lemma1}}Under the assumptions of Theorem $\ref{TH1},$ we have
\[
-\left(  1+\frac{1}{\alpha}\right)  ^{-1}\pi_{k}^{\left(  1\right)  }\left(
\gamma_{1}^{\ast}\right)  =\int_{1}^{\infty}\phi_{\alpha,\gamma_{1}^{\ast}%
}\left(  x\right)  \left\{  \frac{\overline{F}_{n}^{\left(  NA\right)
}\left(  Z_{n-k:n}x\right)  }{\overline{F}_{n}^{\left(  NA\right)  }\left(
Z_{n-k:n}\right)  }-x^{-1/\gamma_{1}^{\ast}}\right\}  dx,
\]
where $\phi_{\alpha,\gamma_{1}^{\ast}}\left(  x\right)  $ is as in $\left(
\ref{phi}\right)  .$ Moreover $\pi_{k}^{\left(  1\right)  }\left(  \gamma
_{1}^{\ast}\right)  \overset{\mathbf{P}}{\rightarrow}0,$ as $n\rightarrow
\infty.$
\end{lemma}

\begin{proof}
We work under the conditions of Theorem $\ref{TH1},$ which assume that the
cdfs $F$ and $G$ exhibit Pareto-type tail behavior with positive tail indices
$\gamma_{1}$ and $\gamma_{2},$ that the usual assumptions on the intermediate
sequence $k_{n}$ hold, and that the proportion satisfies $p>1/2.\smallskip$

Using the notations introduced in $\left(  \ref{pim}\right)  $ and $\left(
\ref{psim-Am}\right)  ,$ we can write
\begin{equation}
-\pi_{k}^{\left(  1\right)  }\left(  \gamma_{1}^{\ast}\right)  =\left(
1+\dfrac{1}{\alpha}\right)  A_{k,\alpha,\gamma_{1}^{\ast}}^{\left(  1\right)
}-\int_{1}^{\infty}\Psi_{\gamma_{1}^{\ast},\alpha+1}^{\left(  1\right)
}\left(  x\right)  dx,\label{f1}%
\end{equation}
where
\[
A_{k,\alpha,\gamma_{1}^{\ast}}^{\left(  1\right)  }=\sum_{i=1}^{k}\frac
{\delta_{\left[  n-i+1:n\right]  }}{i}\dfrac{\overline{F}_{n}^{\left(
NA\right)  }\left(  Z_{n-i+1:n}\right)  }{\overline{F}_{n}^{\left(  NA\right)
}\left(  Z_{n-k:n}\right)  }\Psi_{\gamma_{1}^{\ast},\alpha}^{\left(  1\right)
}\left(  \frac{Z_{n-i+1:n}}{Z_{n-k:n}}\right)  ,
\]
and $F_{n}^{\left(  NA\right)  }$ denotes the Nelson-Aalen product-limit
estimator of the cdf $F,$ as defined in $\left(  \ref{Nelson}\right)  .$
Furthermore,
\[
\Psi_{\gamma_{1}^{\ast},\alpha}^{\left(  1\right)  }\left(  x\right)  =\left.
\frac{d\ell_{\gamma_{1}}^{\alpha}\left(  x\right)  }{d\gamma_{1}}\right\vert
_{\gamma=\gamma_{1}^{\ast}}=\frac{\alpha}{\gamma_{1}^{\ast2+\alpha}}\frac{\log
x-\gamma_{1}^{\ast}}{x^{\alpha\left(  1+1/\gamma_{1}^{\ast}\right)  }}.
\]
Applying the representation $\left(  \ref{dFn}\right)  ,$ we may equivalently
write:
\[
A_{k,\alpha,\gamma_{1}^{\ast}}^{\left(  1\right)  }=\int_{1}^{\infty}%
\Psi_{\gamma_{1}^{\ast},\alpha}^{\left(  1\right)  }\left(  x\right)
d\frac{F_{n}^{\left(  NA\right)  }\left(  Z_{n-k:n}x\right)  }{\overline
{F}_{n}^{\left(  NA\right)  }\left(  Z_{n-k:n}\right)  }.
\]
This may be rewritten as%
\begin{equation}
A_{k,\alpha,\gamma_{1}^{\ast}}^{\left(  1\right)  }=-\int_{1}^{\infty}%
\phi_{\alpha,\gamma_{1}^{\ast}}\left(  x\right)  \frac{\overline{F}%
_{n}^{\left(  NA\right)  }\left(  Z_{n-k:n}x\right)  }{\overline{F}%
_{n}^{\left(  NA\right)  }\left(  Z_{n-k:n}\right)  }dx,\label{f2}%
\end{equation}
where $\phi_{\alpha,\gamma_{1}^{\ast}}\left(  x\right)  :=d\Psi_{\gamma
_{1}^{\ast},\alpha}^{\left(  1\right)  }\left(  x\right)  /dx.$ On the other
hand, by assertion $\left(  i\right)  $ of Proposition $\ref{prop1},$ we have%
\[%
{\displaystyle\int_{1}^{\infty}}
\Psi_{\gamma_{1}^{\ast},\alpha+1}^{\left(  1\right)  }\left(  x\right)
dx=-\left(  1+\dfrac{1}{\alpha}\right)
{\displaystyle\int_{1}^{\infty}}
\Psi_{\gamma_{1}^{\ast},\alpha}^{\left(  1\right)  }\left(  x\right)
dx^{-1/\gamma_{1}^{\ast}}.
\]
Applying integration by parts to the second integral gives
\[%
{\displaystyle\int_{1}^{\infty}}
\Psi_{\gamma_{1}^{\ast},\alpha}^{\left(  1\right)  }\left(  x\right)
dx^{-1/\gamma_{1}^{\ast}}=-%
{\displaystyle\int_{1}^{\infty}}
\phi_{\alpha,\gamma_{1}^{\ast}}\left(  x\right)  x^{-1/\gamma_{1}^{\ast}}dx.
\]
Therefore
\begin{equation}%
{\displaystyle\int_{1}^{\infty}}
\Psi_{\gamma_{1}^{\ast},\alpha+1}^{\left(  1\right)  }\left(  x\right)
dx=\left(  1+\dfrac{1}{\alpha}\right)
{\displaystyle\int_{1}^{\infty}}
\phi_{\alpha,\gamma_{1}^{\ast}}\left(  x\right)  x^{-1/\gamma_{1}^{\ast}%
}dx.\label{f3}%
\end{equation}
Combining $\left(  \ref{f1}\right)  ,$ $\left(  \ref{f2}\right)  $ and
$\left(  \ref{f3}\right)  ,$ we obtain%
\[
-\left(  1+\frac{1}{\alpha}\right)  ^{-1}\pi_{k}^{\left(  1\right)  }\left(
\gamma_{1}^{\ast}\right)  =\int_{1}^{\infty}\phi_{\alpha,\gamma_{1}^{\ast}%
}\left(  x\right)  \left\{  \frac{\overline{F}_{n}^{\left(  NA\right)
}\left(  Z_{n-k:n}x\right)  }{\overline{F}_{n}^{\left(  NA\right)  }\left(
Z_{n-k:n}\right)  }-x^{-1/\gamma_{1}^{\ast}}\right\}  dx.
\]
We decompose this integral as%
\[
I_{1n}:=\int_{1}^{\infty}\phi_{\alpha,\gamma_{1}^{\ast}}\left(  x\right)
\left\{  \frac{\overline{F}_{n}^{\left(  NA\right)  }\left(  Z_{n-k:n}%
x\right)  }{\overline{F}_{n}^{\left(  NA\right)  }\left(  Z_{n-k:n}\right)
}-\frac{\overline{F}\left(  Z_{n-k:n}x\right)  }{\overline{F}\left(
Z_{n-k:n}\right)  }\right\}  dx,
\]
and%
\[
I_{2n}:=\int_{1}^{\infty}\phi_{\alpha,\gamma_{1}^{\ast}}\left(  x\right)
\left\{  \frac{\overline{F}\left(  Z_{n-k:n}x\right)  }{\overline{F}\left(
Z_{n-k:n}\right)  }-x^{-1/\gamma_{1}^{\ast}}\right\}  dx.
\]
For $I_{1n},$ using assertion $\left(  6.29\right)  $ of \cite{MNS2025}, there
exists a sequence of Wiener processes $\left\{  W_{n}\left(  s\right)  ,\text{
}0\leq s\leq1\right\}  ,$ such that, for sufficiently small $\eta,\epsilon
_{0}>0,$%
\[
\sqrt{k}\left\{  \frac{\overline{F}_{n}^{\left(  NA\right)  }\left(
Z_{n-k:n}x\right)  }{\overline{F}_{n}^{\left(  NA\right)  }\left(
Z_{n-k:n}\right)  }-\frac{\overline{F}\left(  Z_{n-k:n}x\right)  }%
{\overline{F}\left(  Z_{n-k:n}\right)  }\right\}  =J_{n}\left(  x\right)
+o_{\mathbf{P}}\left(  1\right)  x^{\left(  2\eta-p\right)  /\gamma
+\epsilon_{0}},
\]
uniformly for $x\geq1,$ where $J_{n}\left(  x\right)  $ is a centered Gaussian
process as defined: in $\left(  \ref{Jn}\right)  .$ Consequently,%
\[
I_{1n}=o_{\mathbf{P}}\left(  1\right)  \int_{1}^{\infty}J_{n}\left(  x\right)
\phi_{\alpha,\gamma_{1}^{\ast}}\left(  x\right)  dx+o_{\mathbf{P}}\left(
1\right)  \int_{1}^{\infty}x^{\left(  2\eta-p\right)  /\gamma+\epsilon_{0}%
}\phi_{\alpha,\gamma_{1}^{\ast}}\left(  x\right)  dx.
\]
Since $\gamma_{1}<\gamma_{2},$ we have $p>1/2$ and thus $\left(
2\eta-p\right)  /\gamma+\epsilon_{0}<0.$ The second integral is therefore
finite, as it reduces to a linear combination of integrals of the form
\[
\int_{1}^{\infty}x^{-d}\left(  \log x\right)  ^{j}dx,\text{ }j=0,1,
\]
for some $d>0.$ This implies that the second term in $I_{1n}$ is
$o_{\mathbf{P}}\left(  1\right)  .\medskip$

\noindent Next, we show that the expectation of the stochastic integral is
also finite. Observe that%
\begin{equation}
\mathbf{E}\left\vert \int_{1}^{\infty}J_{n}\left(  x\right)  \phi
_{\alpha,\gamma_{1}^{\ast}}\left(  x\right)  dx\right\vert <\int_{1}^{\infty
}\mathbf{E}\left\vert J_{n}\left(  x\right)  \right\vert \left\vert
\phi_{\alpha,\gamma_{1}^{\ast}}\left(  x\right)  \right\vert dx. \label{exp}%
\end{equation}
Since $\left\{  W_{n}\left(  s\right)  ,\text{ }0\leq s\leq1\right\}  $ is a
sequence of Wiener processes, it follows that%
\[
\sqrt{\dfrac{n}{k}}\mathbf{E}\left\vert \mathbf{W}_{n,1}\left(  \dfrac{k}%
{n}s\right)  \right\vert \leq\left(  ps\right)  ^{1/2}\text{ and }\sqrt
{\dfrac{n}{k}}\mathbf{E}\left\vert \mathbf{W}_{n,2}\left(  \dfrac{k}%
{n}s\right)  \right\vert \leq\left(  qs\right)  ^{1/2},
\]
where $q:=1-p.$ It follows that
\[
\mathbf{E}\left\vert J_{n}\left(  x\right)  \right\vert =O\left(  x^{1/\left(
2\gamma_{2}\right)  -1/\left(  2\gamma_{1}\right)  }\right)  ,
\]
uniformly for $x>1.$ Since $\gamma_{1}<\gamma_{2},$ we have $1/\left(
2\gamma_{2}\right)  -1/\left(  2\gamma_{1}\right)  <0,$ and therefore
\[
\mathbf{E}\left\vert J_{n}\left(  x\right)  \right\vert =O\left(  1\right)  ,
\]
uniformly for $x>1.$ On the other hand, we already established that
\[
\int_{1}^{\infty}\left\vert \phi_{\alpha,\gamma_{1}^{\ast}}\left(  x\right)
\right\vert dx<\infty,
\]
implying that the expectation in inequality $\left(  \ref{exp}\right)  $ is
finite. Hence, we conclude that $I_{1n}=o_{\mathbf{P}}\left(  1\right)  .$ For
$I_{2n},$ recall that%
\[
\left\vert I_{2n}\right\vert \leq\int_{1}^{\infty}\left\vert \frac
{\overline{F}\left(  Z_{n-k:n}x\right)  }{\overline{F}\left(  Z_{n-k:n}%
\right)  }-x^{-1/\gamma_{1}^{\ast}}\right\vert \left\vert \phi_{\alpha
,\gamma_{1}^{\ast}}\left(  x\right)  \right\vert dx.
\]
By assumption $\left(  \ref{RVbis}\right)  ,$ the tail of $F$ satisfies
\[
\overline{F}\left(  x\right)  =x^{-1/\gamma_{1}^{\ast}}\ell_{1}\left(
x\right)  ,
\]
where $\ell_{1}\left(  x\right)  $ is a slowly varying function. Then, we can
write%
\[
\frac{\overline{F}\left(  Z_{n-k:n}x\right)  }{\overline{F}\left(
Z_{n-k:n}\right)  }-x^{-1/\gamma_{1}^{\ast}}=x^{-1/\gamma_{1}^{\ast}}\left\{
\frac{\ell_{1}\left(  Z_{n-k:n}x\right)  }{\ell_{1}\left(  Z_{n-k:n}\right)
}-1\right\}  ,
\]
Applying Proposition B. 1.10 in \cite{deHF06} (p. 369), for sufficiently small
$\epsilon>0$, we have uniformly for large $Z_{n-k:n},$
\[
\left\vert \frac{\ell_{1}\left(  Z_{n-k:n}x\right)  }{\ell_{1}\left(
Z_{n-k:n}\right)  }-1\right\vert \leq\epsilon x^{\epsilon},\text{ }x>1.
\]
and thus%
\[
\left\vert \frac{\overline{F}\left(  Z_{n-k:n}x\right)  }{\overline{F}\left(
Z_{n-k:n}\right)  }-x^{-1/\gamma_{1}^{\ast}}\right\vert \leq\epsilon
x^{-1/\gamma_{1}^{\ast}+\epsilon}.
\]
Consequently%
\[
\left\vert I_{2n}\right\vert \leq\epsilon\int_{1}^{\infty}x^{-1/\gamma
_{1}^{\ast}+\epsilon}\left\vert \phi_{\alpha,\gamma_{1}^{\ast}}\left(
x\right)  \right\vert dx.
\]
Since $\phi_{\alpha,\gamma_{1}^{\ast}}\left(  x\right)  $ decays sufficiently
fast and, for $\epsilon>0,$ can be chosen arbitrarily small, the above
integral is finite. Therefore,
\[
I_{2n}=o_{\mathbf{P}}\left(  1\right)  .
\]
Combining the two results, we conclude that%
\[
\pi_{k}^{\left(  1\right)  }\left(  \gamma_{1}^{\ast}\right)  =o_{\mathbf{P}%
}\left(  1\right)  ,
\]
which completes the proof.
\end{proof}

\begin{lemma}
\textbf{\label{lemma2}}Under the assumptions of Lemma $\ref{lemma1},$ we have
\begin{equation}
\pi_{k}^{\left(  2\right)  }\left(  \gamma_{1}^{\ast}\right)
\overset{\mathbf{P}}{\rightarrow}\eta^{\ast},\text{ as }n\rightarrow\infty,
\label{app2}%
\end{equation}
where%
\begin{equation}
\eta^{\ast}:=\frac{1+\alpha}{\gamma_{1}^{\ast2+\alpha}}\frac{\alpha^{2}\left(
1+\gamma_{1}^{\ast}\right)  ^{2}+1}{\left(  \alpha\left(  1+\gamma_{1}^{\ast
}\right)  +1\right)  ^{3}}. \label{eta-start}%
\end{equation}

\end{lemma}

\begin{proof}
In view of assertion $\left(  ii\right)  $ of Proposition $\ref{prop1},$ we
have the identity:%
\[
\int_{1}^{\infty}\Psi_{\gamma_{1}^{\ast},\alpha+1}^{\left(  2\right)  }\left(
x\right)  dx=-\left(  1+\frac{1}{\alpha}\right)  \int_{1}^{\infty}\Psi
_{\gamma_{1}^{\ast},\alpha}^{\left(  2\right)  }\left(  x\right)
dx^{-1/\gamma_{1}^{\ast}}+\eta^{\ast},
\]
which implies that%
\[
\left(  1+\frac{1}{\alpha}\right)  ^{-1}\left(  \pi_{k}^{\left(  2\right)
}\left(  \gamma_{1}^{\ast}\right)  -\eta^{\ast}\right)  =-\int_{1}^{\infty
}\Psi_{\gamma_{1}^{\ast},\alpha}^{\left(  2\right)  }\left(  x\right)
dx^{-1/\gamma_{0}}-A_{k,\alpha,\gamma_{1}^{\ast}}^{\left(  2\right)  }.
\]
Following the same line of reasoning as in Lemma $\ref{lemma1},$ we obtain
\[
-\left(  1+\frac{1}{\alpha}\right)  ^{-1}\left(  \pi_{k}^{\left(  2\right)
}\left(  \gamma_{1}^{\ast}\right)  -\eta^{\ast}\right)  =\int_{1}^{\infty}%
\phi_{\alpha,\gamma_{1}^{\ast}}^{\left(  2\right)  }\left(  x\right)  \left\{
\frac{\overline{F}_{n}^{\left(  NA\right)  }\left(  Z_{n-k:n}x\right)
}{\overline{F}_{n}^{\left(  NA\right)  }\left(  Z_{n-k:n}\right)
}-x^{-1/\gamma_{1}^{\ast}}\right\}  dx,
\]
where $\phi_{\alpha,\gamma_{1}^{\ast}}^{\left(  2\right)  }\left(  x\right)
:=d\Psi_{\gamma_{1}^{\ast},\alpha}^{\left(  2\right)  }\left(  x\right)
/dx.\medskip$

\noindent It is also straightforward to verify that
\[
\int_{1}^{\infty}x^{\epsilon}\left\vert \phi_{\alpha,\gamma_{1}^{\ast}%
}^{\left(  2\right)  }\left(  x\right)  \right\vert dx<\infty,
\]
for some $\epsilon>0.$ By applying Proposition $\ref{prop2},$ and the same
arguments as in Lemma $\ref{lemma1}$, we conclude that\textbf{ }%
\[
\pi_{k}^{\left(  2\right)  }\left(  \gamma_{1}^{\ast}\right)  -\eta^{\ast
}=o_{\mathbf{P}}\left(  1\right)  ,
\]
which establishes the result.
\end{proof}

\begin{lemma}
\textbf{\label{lemma3}}Given a sequence of rvs $\widetilde{\gamma}_{1}$ taking
values in a neighborhood of $\gamma_{1}^{\ast},$ that is, such that for some
$\epsilon>0,$ $\left\vert \widetilde{\gamma}_{1}-\gamma_{1}^{\ast}\right\vert
\leq\epsilon$ with probability tending to one, it holds that
\[
\pi_{k}^{\left(  3\right)  }\left(  \widetilde{\gamma}_{1}\right)
=O_{\mathbf{P}}\left(  1\right)  ,\text{ as }n\rightarrow\infty.
\]

\end{lemma}

\begin{proof}
First, we show that\textbf{ }$\pi_{k}^{\left(  3\right)  }\left(  \gamma
_{1}^{\ast}\right)  =O_{\mathbf{P}}\left(  1\right)  .$\textbf{ }Recall
that\textbf{ }%
\begin{align*}
\pi_{k}^{\left(  3\right)  }\left(  \gamma_{1}^{\ast}\right)   &  =\int%
_{1}^{\infty}\Psi_{\gamma_{1}^{\ast},\alpha+1}^{\left(  3\right)  }\left(
x\right)  dx\\
&  -\left(  1+1/\alpha\right)  \sum_{i=1}^{k}\frac{\delta_{\left[
n-i+1:n\right]  }}{i}\dfrac{\overline{F}_{n}^{\left(  NA\right)  }\left(
Z_{n-i+1:n}\right)  }{\overline{F}_{n}^{\left(  NA\right)  }\left(
Z_{n-k:n}\right)  }\Psi_{\gamma_{1}^{\ast},\alpha}^{\left(  3\right)  }\left(
\frac{Z_{n-i+1:n}}{Z_{n-k:n}}\right)  .
\end{align*}
By Proposition\textbf{ }$\ref{prop2},$\textbf{ }we can rewrite this as\textbf{
}%
\[
\pi_{k}^{\left(  3\right)  }\left(  \gamma_{1}^{\ast}\right)  =\left(
1+\frac{1}{\alpha}\right)  A_{k,\alpha,\gamma_{1}^{\ast}}^{\left(  3\right)
}+\zeta^{\ast},
\]
where\textbf{ }%
\[
A_{k,\alpha,\gamma_{1}^{\ast}}^{\left(  3\right)  }:=-\int_{1}^{\infty}%
\Psi_{\gamma_{1}^{\ast},\alpha}^{\left(  3\right)  }\left(  x\right)
dx^{-1/\gamma_{1}^{\ast}}-\sum_{i=1}^{k}\frac{\delta_{\left[  n-i+1:n\right]
}}{i}\dfrac{\overline{F}_{n}^{\left(  NA\right)  }\left(  Z_{n-i+1:n}\right)
}{\overline{F}_{n}^{\left(  NA\right)  }\left(  Z_{n-k:n}\right)  }%
\Psi_{\gamma_{1}^{\ast},\alpha}^{\left(  3\right)  }\left(  \frac{Z_{n-i+1:n}%
}{Z_{n-k:n}}\right)  .
\]
Using arguments similar to Lemma $\ref{lemma2},$ we can show that
$A_{k,\alpha,\gamma_{1}^{\ast}}^{\left(  3\right)  }\overset{\mathbf{P}%
}{\rightarrow}0.$ Hence $\pi_{k}^{\left(  3\right)  }\left(  \gamma_{1}^{\ast
}\right)  \overset{\mathbf{P}}{\rightarrow}\zeta^{\ast}<\infty,$ implying
$\pi_{k}^{\left(  3\right)  }\left(  \gamma_{1}^{\ast}\right)  =O_{\mathbf{P}%
}\left(  1\right)  .$ Next, we prove that $\pi_{k}^{\left(  3\right)  }\left(
\widetilde{\gamma}_{1}\right)  -\pi_{k}^{\left(  3\right)  }\left(  \gamma
_{1}^{\ast}\right)  =O_{\mathbf{P}}\left(  1\right)  ,$ as $n\rightarrow
\infty.$ Indeed, write
\begin{align*}
&  \pi_{k}^{\left(  3\right)  }\left(  \widetilde{\gamma}_{1}\right)  -\pi
_{k}^{\left(  3\right)  }\left(  \gamma_{1}^{\ast}\right)  \medskip\\
&  =\int_{1}^{\infty}\left\{  \Psi_{\widetilde{\gamma}_{1},\alpha+1}^{\left(
3\right)  }\left(  x\right)  -\Psi_{\gamma_{1}^{\ast},\alpha+1}^{\left(
3\right)  }\left(  x\right)  \right\}  dx\medskip\\
&  -\left(  1+1/\alpha\right)  \sum_{i=1}^{k}\frac{\delta_{\left[
n-i+1:n\right]  }}{i}\dfrac{\overline{F}_{n}^{\left(  NA\right)  }\left(
Z_{n-i+1:n}\right)  }{\overline{F}_{n}^{\left(  NA\right)  }\left(
Z_{n-k:n}\right)  }\\
&  \ \ \ \ \ \ \ \ \ \ \ \ \ \ \ \ \ \ \ \ \ \times\left\{  \Psi
_{\widetilde{\gamma}_{1},\alpha}^{\left(  3\right)  }\left(  \frac
{Z_{n-i+1:n}}{Z_{n-k:n}}\right)  -\Psi_{\gamma_{1}^{\ast},\alpha}^{\left(
3\right)  }\left(  \frac{Z_{n-i+1:n}}{Z_{n-k:n}}\right)  \right\}  .
\end{align*}
By the mean value theorem, there exists a value $\overline{\gamma}_{1}$ lying
between $\widetilde{\gamma}_{1}$ and $\gamma_{1}^{\ast}$ such that
\[
\int_{1}^{\infty}\left\{  \Psi_{\widetilde{\gamma}_{1},\alpha+1}^{\left(
3\right)  }\left(  x\right)  -\Psi_{\gamma_{1}^{\ast},\alpha+1}^{\left(
3\right)  }\left(  x\right)  \right\}  dx=\left(  \widetilde{\gamma}%
-\gamma_{1}^{\ast}\right)  \int_{1}^{\infty}\frac{d^{4}}{d\gamma^{4}}%
\ell_{\overline{\gamma}_{1}}^{\alpha+1}\left(  x\right)  dx.
\]
From Proposition 8.2 in \cite{MNM2025}, the integral on the right-hand side is
stochastically finite. Since $\widetilde{\gamma}_{1}$ lies in a neighborhood
of $\gamma_{1}^{\ast},$ $\widetilde{\gamma}_{1}-\gamma_{1}^{\ast
}=O_{\mathbf{P}}\left(  1\right)  ,$ and therefore the integral difference is
$O_{\mathbf{P}}\left(  1\right)  .\medskip$

\noindent Similarly, the summation term can be expressed as a Stieltjes
integral:
\[%
\begin{array}
[c]{cl}
&
{\displaystyle\sum\limits_{i=1}^{k}}
\dfrac{\delta_{\left[  n-i+1:n\right]  }}{i}\dfrac{\overline{F}_{n}^{\left(
NA\right)  }\left(  Z_{n-i+1:n}\right)  }{\overline{F}_{n}^{\left(  NA\right)
}\left(  Z_{n-k:n}\right)  }\left\{  \Psi_{\widetilde{\gamma}_{1},\alpha
}^{\left(  3\right)  }\left(  \dfrac{Z_{n-i+1:n}}{Z_{n-k:n}}\right)
-\Psi_{\gamma_{1}^{\ast},\alpha+1}^{\left(  3\right)  }\left(  \dfrac
{Z_{n-i+1:n}}{Z_{n-k:n}}\right)  \right\}  \bigskip\\
= &
{\displaystyle\int_{1}^{\infty}}
\left\{  \Psi_{\widetilde{\gamma}_{1},\alpha}^{\left(  3\right)  }\left(
x\right)  -\Psi_{\gamma_{1}^{\ast},\alpha+1}^{\left(  3\right)  }\left(
x\right)  \right\}  d\dfrac{\overline{F}_{n}^{\left(  NA\right)  }\left(
Z_{n-k:n}x\right)  }{\overline{F}_{n}^{\left(  NA\right)  }\left(
Z_{n-k:n}\right)  }.
\end{array}
\]
Applying the mean value theorem and integration by parts, this integral is
also\textbf{ }%
\[
\left(  \widetilde{\gamma}_{1}-\gamma_{1}^{\ast}\right)  \int_{1}^{\infty}%
\Psi_{\overline{\overline{\gamma}}_{1},\alpha}^{\left(  4\right)  }\left(
x\right)  d\dfrac{\overline{F}_{n}^{\left(  NA\right)  }\left(  Z_{n-k:n}%
x\right)  }{\overline{F}_{n}^{\left(  NA\right)  }\left(  Z_{n-k:n}\right)
},
\]
where $\overline{\overline{\gamma}}_{1}$ lies between $\widetilde{\gamma}_{1}$
and $\gamma_{1}^{\ast}.$ By an integration by parts, we obtain%
\begin{align*}
&  \int_{1}^{\infty}\Psi_{\overline{\overline{\gamma}}_{1},\alpha}^{\left(
4\right)  }\left(  x\right)  d\dfrac{\overline{F}_{n}^{\left(  NA\right)
}\left(  Z_{n-k:n}x\right)  }{\overline{F}_{n}^{\left(  NA\right)  }\left(
Z_{n-k:n}\right)  }\\
&  =\Psi_{\overline{\overline{\gamma}}_{1},\alpha}^{\left(  4\right)  }\left(
1\right)  -\int_{1}^{\infty}\left\{  \dfrac{\overline{F}_{n}^{\left(
NA\right)  }\left(  Z_{n-k:n}x\right)  }{\overline{F}_{n}^{\left(  NA\right)
}\left(  Z_{n-k:n}\right)  }-\dfrac{\overline{F}\left(  Z_{n-k:n}x\right)
}{\overline{F}\left(  Z_{n-k:n}\right)  }\right\}  \phi_{\alpha,\overline
{\overline{\gamma}}_{1}}^{\left(  4\right)  }\left(  x\right)  dx\\
&  +\int_{1}^{\infty}\left\{  \dfrac{\overline{F}\left(  Z_{n-k:n}x\right)
}{\overline{F}\left(  Z_{n-k:n}\right)  }-x^{-1/\gamma_{1}}\right\}
\phi_{\alpha,\overline{\overline{\gamma}}_{1}}^{\left(  4\right)  }\left(
x\right)  dx+\int_{1}^{\infty}x^{-1/\gamma_{1}}\phi_{\alpha,\overline
{\overline{\gamma}}_{1}}^{\left(  4\right)  }\left(  x\right)  dx,
\end{align*}
where $\phi_{\alpha,\overline{\overline{\gamma}}_{1}}^{\left(  4\right)
}\left(  x\right)  :=d\Psi_{\overline{\overline{\gamma}}_{11}^{\ast},\alpha
}^{\left(  4\right)  }\left(  x\right)  /dx.$ By arguments similar to those
used earlier, the first, and fourth terms are $O_{\mathbf{P}}\left(  1\right)
$ while the second and third terms are $o_{\mathbf{P}}\left(  1\right)  .$
Hence, the integral is $O_{\mathbf{P}}\left(  1\right)  .\bigskip$

Since for a fixed $\epsilon>0,$ $\left\vert \widetilde{\gamma}_{1}-\gamma
_{1}^{\ast}\right\vert \leq\epsilon,$ is follows that $\widetilde{\gamma}%
_{1}-\gamma_{1}^{\ast}=O_{\mathbf{P}}\left(  1\right)  .$ We conclude that
\[
\pi_{k}^{\left(  3\right)  }\left(  \widetilde{\gamma}_{1}\right)  -\pi
_{k}^{\left(  3\right)  }\left(  \gamma_{1}^{\ast}\right)  =O_{\mathbf{P}%
}\left(  1\right)  ,
\]
which finally implies that
\[
\pi_{k}^{\left(  3\right)  }\left(  \widetilde{\gamma}_{1}\right)
=O_{\mathbf{P}}\left(  1\right)
\]
as claimed.
\end{proof}

\begin{proposition}
\textbf{\label{prop1}}For any $\alpha>0,$ the following identities hold.
\[%
\begin{array}
[c]{l}%
\left(  i\right)  \text{ }%
{\displaystyle\int_{1}^{\infty}}
\Psi_{\gamma_{1}^{\ast},\alpha+1}^{\left(  1\right)  }\left(  x\right)
dx=-\left(  1+\dfrac{1}{\alpha}\right)
{\displaystyle\int_{1}^{\infty}}
\Psi_{\gamma_{1}^{\ast},\alpha}^{\left(  1\right)  }\left(  x\right)
dx.\medskip\\
\left(  ii\right)  \text{ }%
{\displaystyle\int_{1}^{\infty}}
\Psi_{\gamma_{1}^{\ast},\alpha+1}^{\left(  2\right)  }\left(  x\right)
dx=-\left(  1+\dfrac{1}{\alpha}\right)
{\displaystyle\int_{1}^{\infty}}
\Psi_{\gamma_{1}^{\ast},\alpha}^{\left(  2\right)  }\left(  x\right)
dx^{-1/\gamma_{1}^{\ast}}+\eta^{\ast},\\
\text{where }\eta^{\ast}>0\text{ is as in }\left(  \ref{eta-start}\right)  .\\
\left(  iii\right)  \text{ }\int_{1}^{\infty}\Psi_{\gamma_{1}^{\ast},\alpha
+1}^{\left(  3\right)  }\left(  x\right)  dx=-\left(  1+\dfrac{1}{\alpha
}\right)  \int_{1}^{\infty}\Psi_{\gamma_{1}^{\ast},\alpha}^{\left(  3\right)
}\left(  x\right)  dx^{-1/\gamma_{1}^{\ast}}+\zeta^{\ast},\\
\text{where }\zeta^{\ast}\text{ denotes a finite constant.}%
\end{array}
\]

\end{proposition}

\begin{proof}
Let $J$ be continuous, nonincreasing, and nonnegative on $\left(  0,1\right)
,$ and satisfies $\int_{0}^{1}J\left(  s\right)  ds=1.$ From Proposition 8.1
of \cite{MNM2025}, for every $\alpha>0,$ we have
\begin{align*}
\Psi_{\gamma_{1}^{\ast},\alpha+1}^{(1)}(x)  &  =(1+1/\alpha)\ell_{\gamma
_{1}^{\ast},J}(x)\Psi_{\gamma_{1}^{\ast},\alpha}^{(1)}(x),\\
\Psi_{\gamma_{1}^{\ast},\alpha+1}^{(2)}(x)  &  =(1+1/\alpha)\ell_{\gamma
_{1}^{\ast},J}(x)\Psi_{\gamma_{1}^{\ast},\alpha}^{(2)}(x)+(1+\alpha)\left(
\Psi_{\gamma_{1}^{\ast},1}^{(1)}(x)\right)  ^{2}\ell_{\gamma_{1}^{\ast}%
,J}^{\alpha-1}(x),\\
\Psi_{\gamma_{1}^{\ast},\alpha+1}^{(3)}(x)  &  =(1+1/\alpha)\ell_{\gamma
_{1}^{\ast},J}(x)\Psi_{\gamma_{1}^{\ast},\alpha}^{(3)}(x)+g_{\gamma_{1}^{\ast
},\alpha}(x),
\end{align*}
for some integrable function $x\mapsto g_{\gamma_{1}^{\ast},\alpha}(x),$
where
\[
\ell_{\gamma_{1}^{\ast},J}(x):=J(x^{-1/\gamma_{1}^{\ast}})\ell_{\gamma
_{1}^{\ast}}(x),\text{ }x>1.
\]
Integrating each equality over $(1,\infty)$ yields the identities in
Proposition \ref{prop1}, with
\[
0<\eta^{\ast}=(1+\alpha)\int_{1}^{\infty}\left(  \Psi_{\gamma_{1}^{\ast}%
,1}^{(1)}(x)\right)  ^{2}\ell_{\gamma_{1}^{\ast}}^{\alpha-1}(x)\,dx<\infty
,\quad\zeta^{\ast}=\int_{1}^{\infty}g_{\gamma_{1}^{\ast},\alpha}%
(x)\,dx<\infty.
\]
In the present paper, we take $J(x)=1$, which reduces $\ell_{\gamma_{1}^{\ast
},J}(x)$ to $\ell_{\gamma_{1}^{\ast}}(x).$ This choice trivially satisfies all
the above regularity conditions, so the identities remain valid for the strict
Pareto law. This completes the proof of Proposition~\ref{prop1}.
\end{proof}

\begin{proposition}
\label{prop2} The asymptotic variance%
\begin{equation}%
\begin{array}
[c]{l}%
\sigma^{2}:=%
{\displaystyle\int\nolimits_{1}^{\infty}}
{\displaystyle\int\nolimits_{1}^{\infty}}
\min(x^{-1/\gamma^{\ast}},y^{-1/\gamma^{\ast}})K(x,y)\,dx\,dy\\
\ \ \ \ \ \ \ \ \ \ \ \ \ \ \ -2%
{\displaystyle\int\nolimits_{1}^{\infty}}
x^{-1/\gamma^{\ast}}\Delta_{1}(x)\,dx+p\left(
{\displaystyle\int\nolimits_{1}^{\infty}}
x^{-1/\gamma_{1}^{\ast}}\phi_{\alpha,\gamma_{1}^{\ast}}(x)\,dx\right)  ^{2},
\end{array}
\label{sigma}%
\end{equation}
where
\begin{align*}
K(x,y)  &  :=p\Delta_{1}(x)\Delta_{1}(y)+\frac{q}{\gamma_{1}^{\ast}}\Delta
_{2}(x)\Delta_{2}(y),\\
\Delta_{1}(x)  &  :=x^{q/\gamma^{\ast}}\phi_{\alpha,\gamma_{1}^{\ast}%
}(x)-qx^{1/\gamma^{\ast}}V(x),\quad\Delta_{2}(x):=x^{1/\gamma^{\ast}}V(x),\\
V(x)  &  :=\int_{x}^{\infty}t^{-1/\gamma_{1}^{\ast}}\phi_{\alpha,\gamma
_{1}^{\ast}}(t)\,dt.
\end{align*}
Here $\phi_{\alpha,\gamma_{1}^{\ast}}$ is defined in \ref{phi}. Moreover,
$\gamma^{\ast}$ and $\gamma_{1}^{\ast}$ denote the true values of the tail
indices $\gamma$ and $\gamma_{1}$ associated with the cdfs $\overline{H}$ and
$\overline{F}$ respectively.
\end{proposition}

\begin{proof}
Recalling that $q=1-p$ and $\gamma_{2}^{\ast}=\gamma^{\ast}/q,$ and the
expression of $J_{n}(x)$ given in $\left(  \ref{Jn}\right)  ,$ we decompose
\[
\mathbb{I}_{n}:=\int_{1}^{\infty}J_{n}(x)\phi_{\alpha,\gamma_{1}^{\ast}%
}(x)\,dx=\mathbb{I}_{n1}+\mathbb{I}_{n2},
\]
where
\[
\mathbb{I}_{n1}:=\int_{1}^{\infty}J_{n1}(x)\phi_{\alpha,\gamma_{1}^{\ast}%
}(x)\,dx\text{ and }\mathbb{I}_{n2}:=\int_{1}^{\infty}J_{n2}(x)\phi
_{\alpha,\gamma_{1}^{\ast}}(x)\,dx.
\]
with%
\[
J_{n1}\left(  x\right)  :=\sqrt{\frac{n}{k}}\left\{  x^{q/\gamma^{\ast}%
}\mathbf{W}_{n,1}\left(  \frac{k}{n}x^{-1/\gamma^{\ast}}\right)
-x^{-1/\gamma_{1}^{\ast}}\mathbf{W}_{n,1}\left(  \frac{k}{n}\right)  \right\}
,
\]
and
\[
J_{n2}\left(  x\right)  :=\dfrac{x^{-1/\gamma_{1}^{\ast}}}{\gamma^{\ast}}%
\sqrt{\dfrac{n}{k}}%
{\displaystyle\int_{1}^{x}}
u^{1/\gamma^{\ast}-1}\left\{  p\mathbf{W}_{n,2}\left(  \dfrac{k}%
{n}u^{-1/\gamma^{\ast}}\right)  -q\mathbf{W}_{n,1}\left(  \dfrac{k}%
{n}u^{-1/\gamma^{\ast}}\right)  \right\}  du.
\]
It is straightforward to verify that
\[
\mathbb{I}_{n1}=\sqrt{\frac{n}{k}}\int_{1}^{\infty}x^{q/\gamma^{\ast}}%
W_{n,1}\left(  \frac{k}{n}x^{-1/\gamma^{\ast}}\right)  \phi_{\alpha,\gamma
_{1}^{\ast}}(x)\,dx-a\sqrt{\frac{n}{k}}W_{n,1}\left(  \frac{k}{n}\right)  ,
\]
where $a:=\int_{1}^{\infty}x^{-q/\gamma^{\ast}}\phi_{\alpha,\gamma_{1}^{\ast}%
}(x)\,dx.$ Integrating $\mathbb{I}_{n2}$ by parts, yields
\[
\mathbb{I}_{n2}=\frac{1}{\gamma_{1}^{\ast}}\sqrt{\frac{n}{k}}\int_{1}^{\infty
}x^{q/\gamma^{\ast}}\left\{  pW_{n,2}\left(  \frac{k}{n}x^{-1/\gamma^{\ast}%
}\right)  -qW_{n,1}\left(  \frac{k}{n}x^{-1/\gamma^{\ast}}\right)  \right\}
V(x)\,dx.
\]
Hence, $\mathbb{I}_{n}$ can be rewritten as%
\begin{align*}
\mathbb{I}_{n1} &  =\sqrt{\frac{n}{k}}\int_{1}^{\infty}W_{n,1}\left(  \frac
{k}{n}x^{-1/\gamma_{1}^{\ast}}\right)  \Delta_{1}(x)\,dx-a\sqrt{\frac{n}{k}%
}W_{n,1}\left(  \frac{k}{n}\right)  ,\\
\mathbb{I}_{n2} &  =\frac{1}{\gamma_{1}^{\ast}}\sqrt{\frac{n}{k}}\int%
_{1}^{\infty}W_{n,2}\left(  \frac{k}{n}x^{-1/\gamma^{\ast}}\right)  \Delta
_{2}(x)\,dx.
\end{align*}
Since $W_{n,1}$ and $W_{n,2}$ are centered and independent Gaussian processes,
the rvs $\mathbb{I}_{n1},$ $\mathbb{I}_{n2},$ and their product $\mathbb{I}%
_{n1}\times\mathbb{I}_{n2}$ are centered. Taking expectations, we obtain
\begin{align*}
\mathbf{E}[\mathbb{I}_{n1}^{2}] &  =p\int_{1}^{\infty}\int_{1}^{\infty}%
\min(x^{-1/\gamma^{\ast}},y^{-1/\gamma^{\ast}})\Delta_{1}(x)\Delta
_{1}(y)\,dx\,dy+pa^{2}-2\int_{1}^{\infty}x^{-1/\gamma^{\ast}}\Delta
_{1}(x)\,dx.\\
\mathbf{E}[\mathbb{I}_{n2}^{2}] &  =\frac{q}{\gamma_{1}^{\ast}}\int%
_{1}^{\infty}\int_{1}^{\infty}\min(x^{-1/\gamma_{1}^{\ast}},y^{-1/\gamma
_{1}^{\ast}})\Delta_{2}(x)\Delta_{2}(y)\,dx\,dy,
\end{align*}
The sum of these two expressions yields the asymptotic variance formula
$\sigma^{2}$ given in $\left(  \ref{sigma}\right)  ,$ which completes the proof.
\end{proof}

\newpage

\section{\textbf{Appendix B: Simulation results}}%

\begin{figure}[h]%
\centering
\includegraphics[
height=3.0926in,
width=4.6345in
]%
{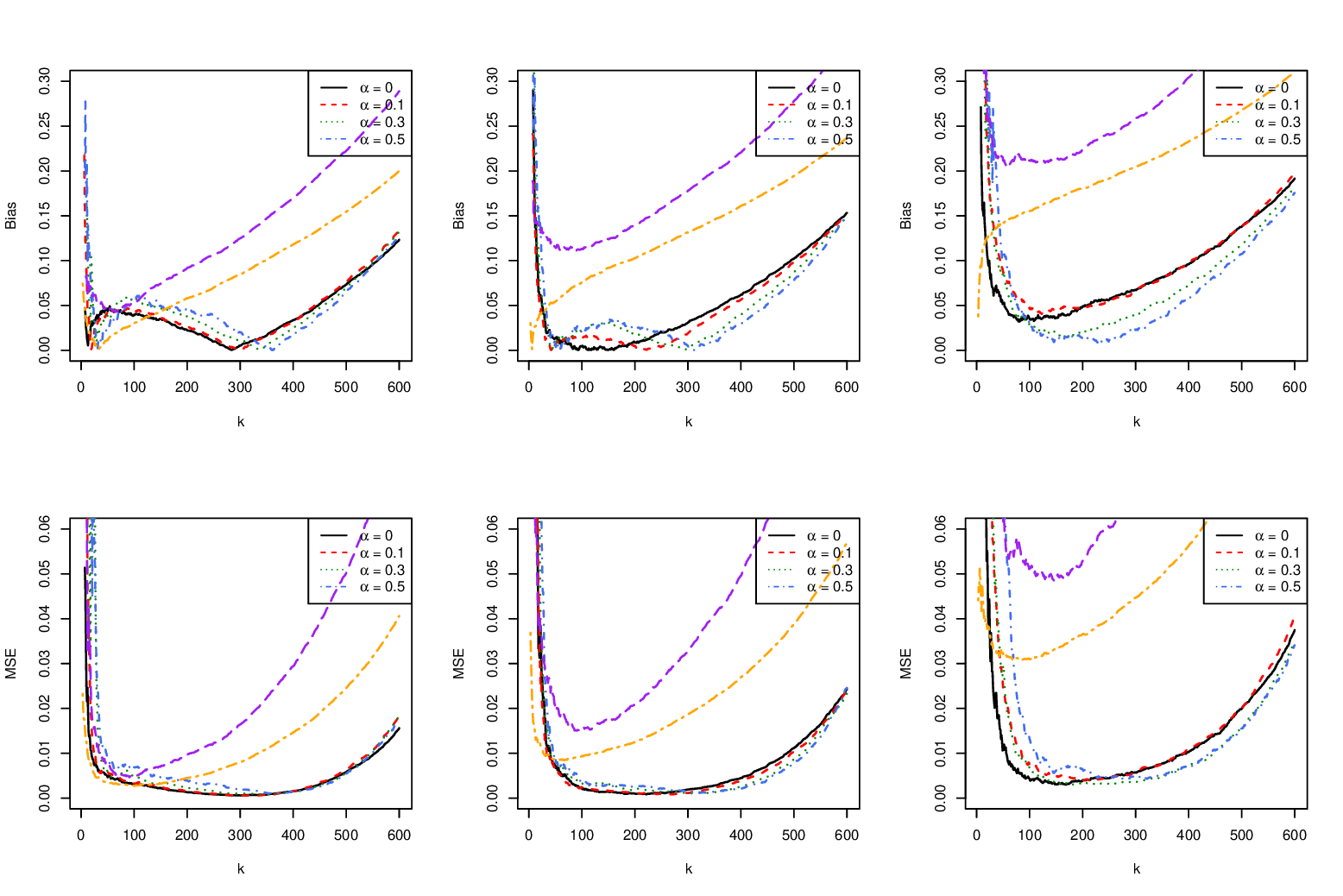}%
\caption{Bias (top panel) and MSE (bottom panel) of the MDPD tail index
estimator $\protect\widehat{\gamma}_{1,\alpha},$ together with
$\protect\widehat{\gamma}_{1}^{(EFG)}$ (long-dashed purple line) and
$\protect\widehat{\gamma}_{1}^{(W)}$\ (two-dashed orange line) computed from
$2000$ Monte Carlo samples of size $1000$ genrated from Scenario \textbf{S1},
with $\gamma_{1}=0.3,$ $\gamma_{c}=0.6$ and $p=0.55.$ Contamination is
introduced before the censoring mechanism, with $\epsilon=0$ (left panel),
$\epsilon=0.15$ (middle panel), and $\epsilon=0.40$ (right panel).}%
\label{fig1}%
\end{figure}
%

\begin{figure}[h]%
\centering
\includegraphics[
height=3.0926in,
width=4.6345in
]%
{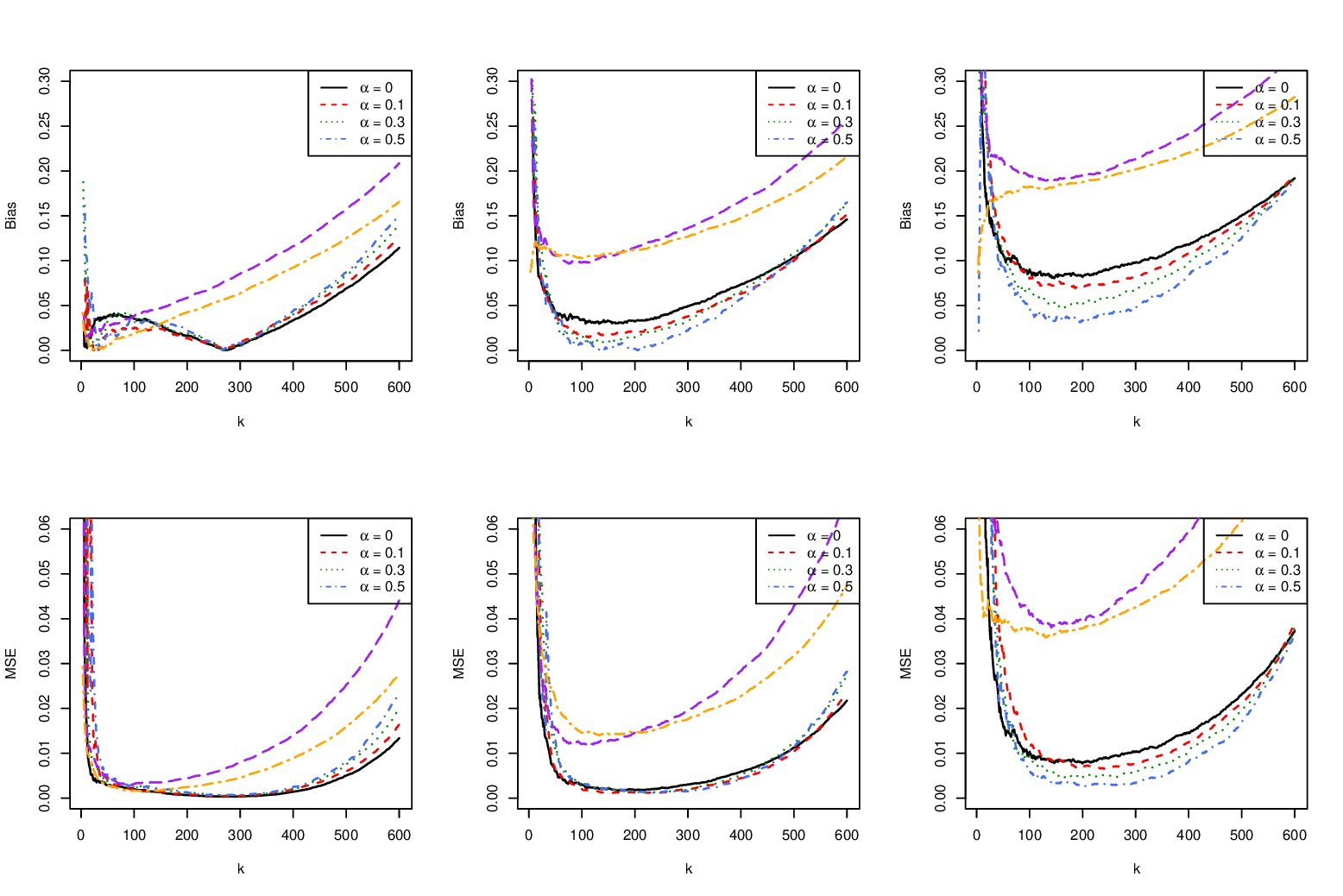}%
\caption{Bias (top panel) and MSE (bottom panel) of the MDPD tail index
estimator $\protect\widehat{\gamma}_{1,\alpha},$ together with
$\protect\widehat{\gamma}_{1}^{(EFG)}$ (long-dashed purple line) and
$\protect\widehat{\gamma}_{1}^{(W)}$\ (two-dashed orange line) computed from
$2000$ Monte Carlo samples of size $1000$ genrated from Scenario \textbf{S1},
with $\gamma_{1}=0.3,$ $\gamma_{c}=0.6$ and $p=0.7.$ Contamination is
introduced before the censoring mechanism, with $\epsilon=0$ (left panel),
$\epsilon=0.15$ (middle panel), and $\epsilon=0.40$ (right panel).}%
\label{fig2}%
\end{figure}
%

\begin{figure}[h]%
\centering
\includegraphics[
height=3.0926in,
width=4.6345in
]%
{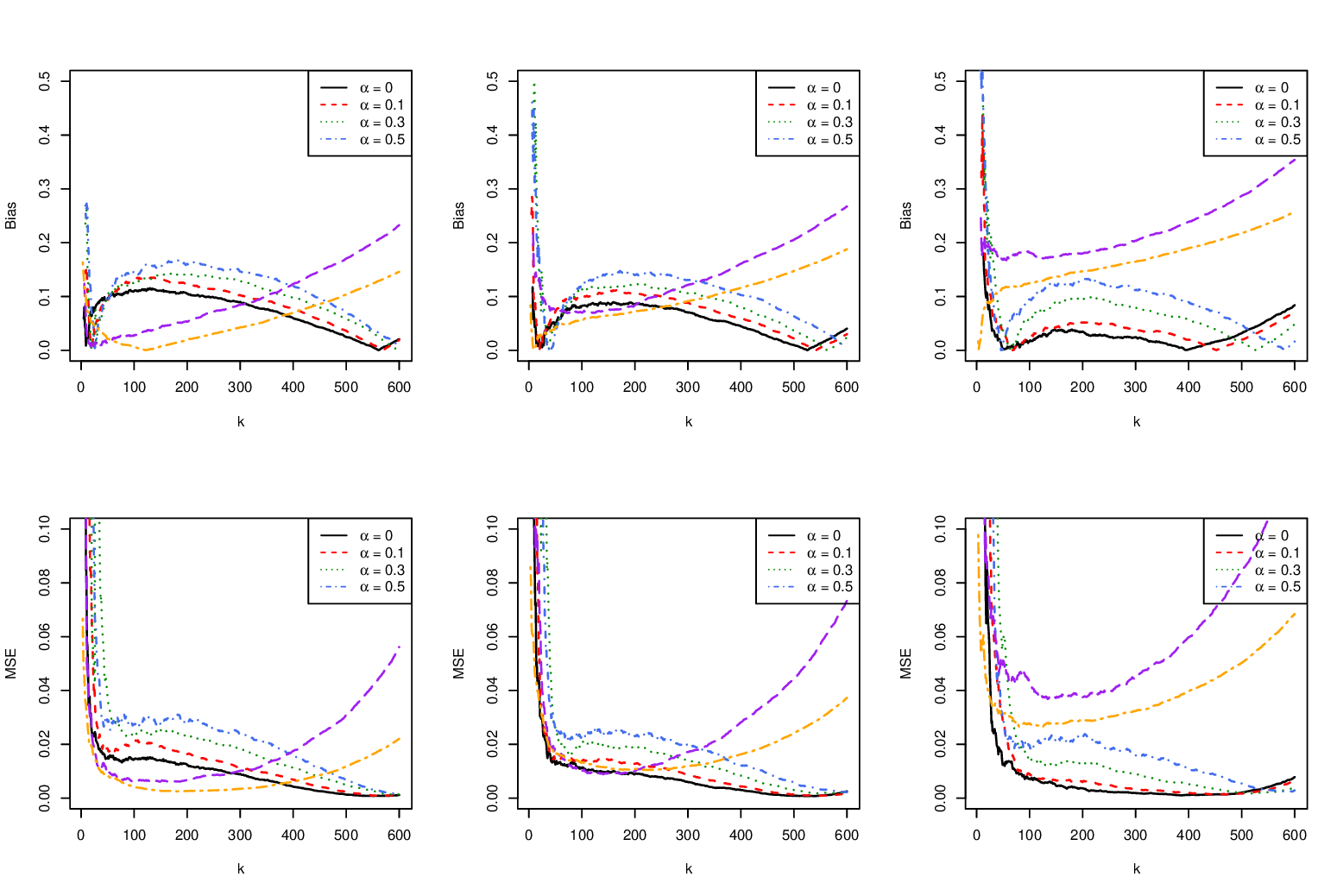}%
\caption{Bias (top panel) and MSE (bottom panel) of the MDPD tail index
estimator $\protect\widehat{\gamma}_{1,\alpha},$ together with
$\protect\widehat{\gamma}_{1}^{(EFG)}$ (long-dashed purple line) and
$\protect\widehat{\gamma}_{1}^{(W)}$\ (two-dashed orange line) computed from
$2000$ Monte Carlo samples of size $1000$ genrated from Scenario \textbf{S1},
with $\gamma_{1}=0.5,$ $\gamma_{c}=0.8$ and $p=0.55.$ Contamination is
introduced before the censoring mechanism, with $\epsilon=0$ (left panel),
$\epsilon=0.15$ (middle panel), and $\epsilon=0.40$ (right panel).}%
\label{fig3}%
\end{figure}
%

\begin{figure}[h]%
\centering
\includegraphics[
height=3.0926in,
width=4.6345in
]%
{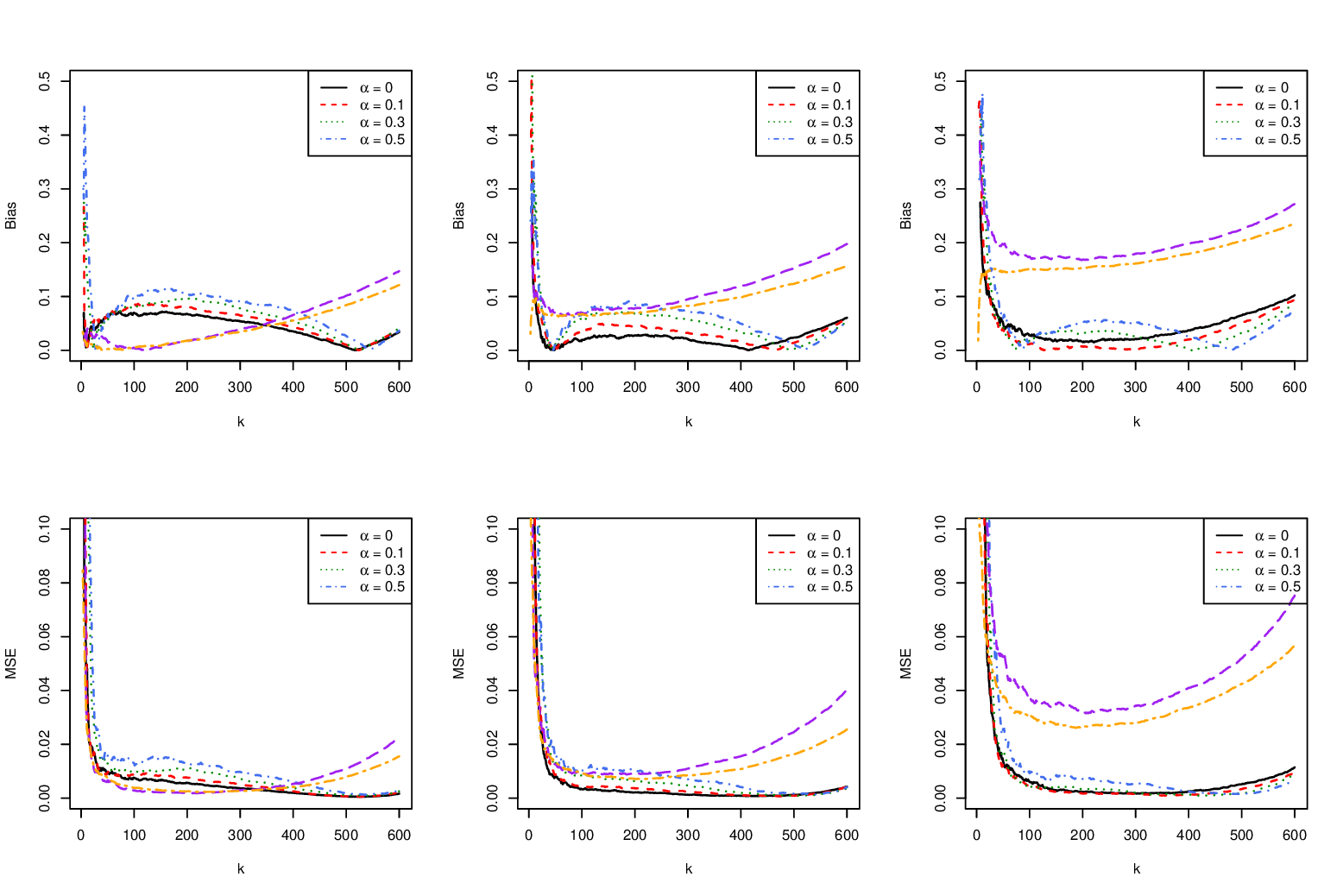}%
\caption{Bias (top panel) and MSE (bottom panel) of the MDPD tail index
estimator $\protect\widehat{\gamma}_{1,\alpha},$ $\protect\widehat{\gamma}%
_{1}^{(EFG)}$ (longdashed purple line) and $\protect\widehat{\gamma}_{1}%
^{(W)}$\ (twodashed orange line) based on $2000$ samples of size $1000$ from
scenario \textbf{S1}, for $\gamma_{1}=0.5,$ $\gamma_{c}=0.8$ and $p=0.7,$
under contamination introduced before the censoring mechanism, with
$\epsilon=0$ (left panel), $\epsilon=0.15$ (middle panel) and $\epsilon=0.40$
(right panel).}%
\label{fig4}%
\end{figure}
%

\begin{figure}[h]%
\centering
\includegraphics[
height=3.0926in,
width=4.6345in
]%
{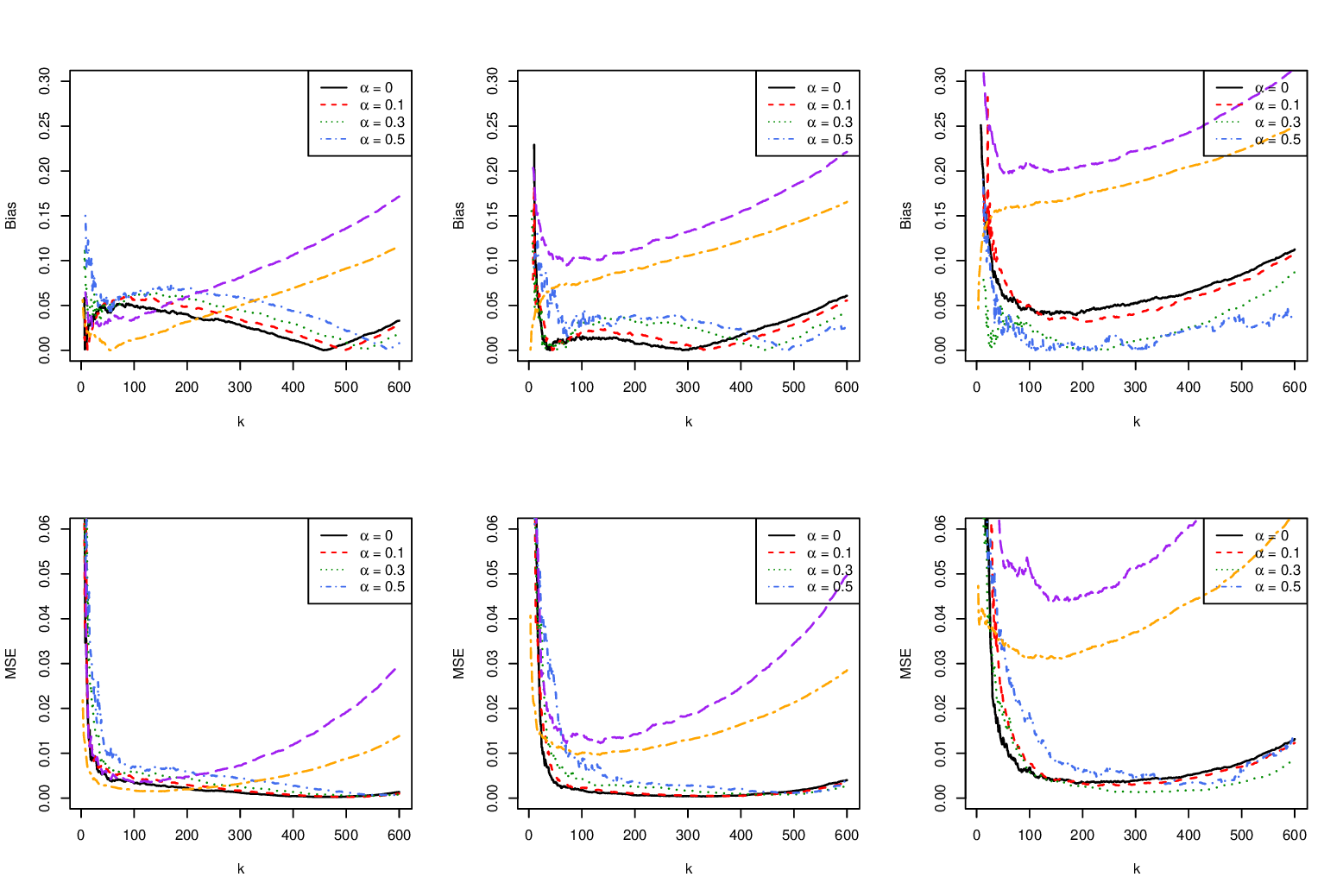}%
\caption{Bias (top panel) and MSE (bottom panel) of the MDPD tail index
estimator $\protect\widehat{\gamma}_{1,\alpha},$ together with
$\protect\widehat{\gamma}_{1}^{(EFG)}$ (long-dashed purple line) and
$\protect\widehat{\gamma}_{1}^{(W)}$\ (two-dashed orange line) computed from
$2000$ Monte Carlo samples of size $1000$ genrated from Scenario \textbf{S2},
with $\gamma_{1}=0.3,$ $\gamma_{c}=0.6$ and $p=0.55.$ Contamination is
introduced before the censoring mechanism, with $\epsilon=0$ (left panel),
$\epsilon=0.15$ (middle panel), and $\epsilon=0.40$ (right panel).}%
\label{fig5}%
\end{figure}
%

\begin{figure}[h]%
\centering
\includegraphics[
height=3.0926in,
width=4.6345in
]%
{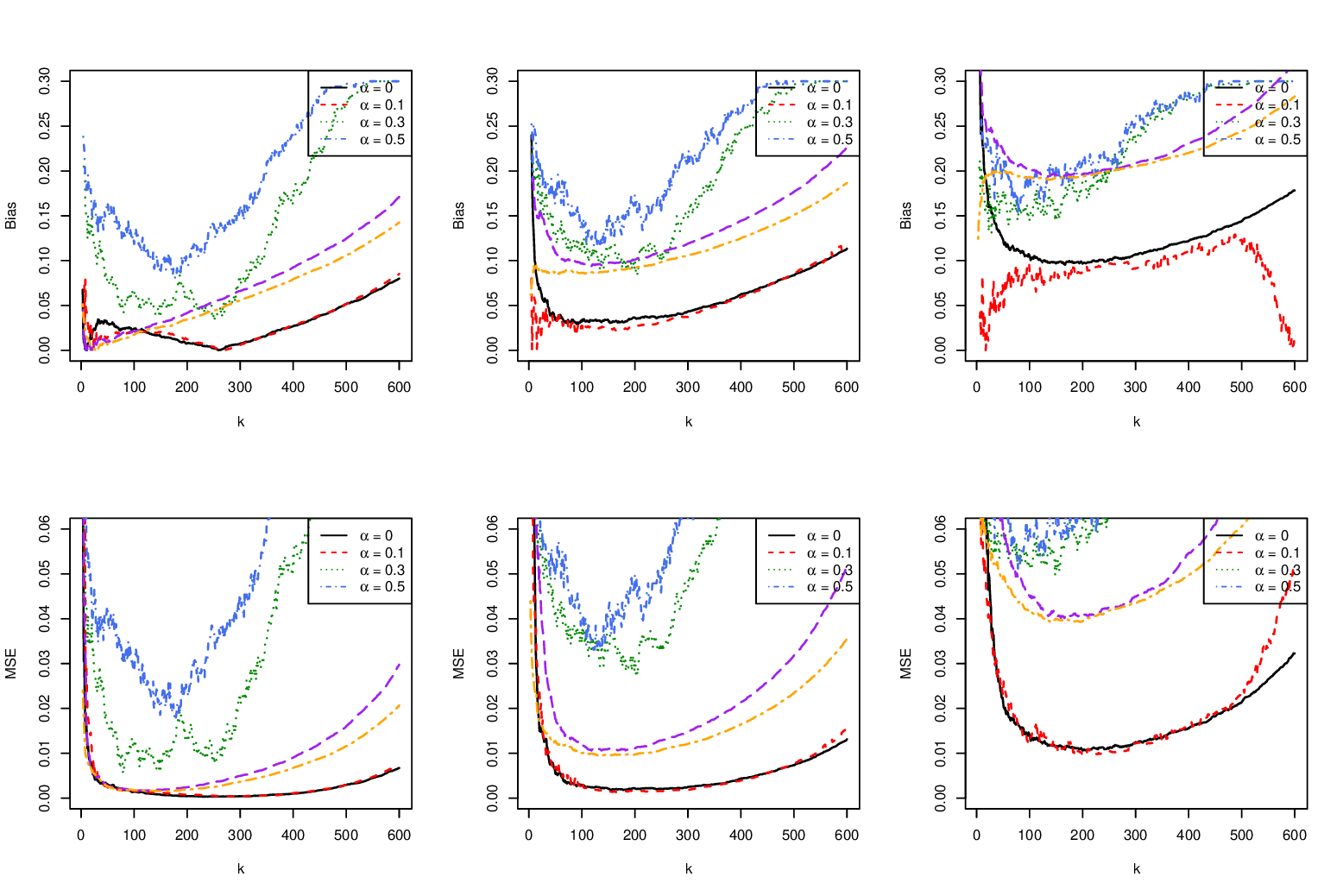}%
\caption{Bias (top) and MSE (bottom) of the MDPD tail index estimator
$\protect\widehat{\gamma}_{1,\alpha},$ together with $\protect\widehat{\gamma
}_{1}^{(EFG)}$ (long-dashed purple line) and $\protect\widehat{\gamma}%
_{1}^{(W)}$\ (two-dashed orange line) computed from $2000$ Monte Carlo samples
of size $1000$ genrated from Scenario \textbf{S2}, with $\gamma_{1}=0.3,$
$\gamma_{c}=0.6$ and $p=0.7.$ Contamination is introduced before the censoring
mechanism, with $\epsilon=0$ (left panel), $\epsilon=0.15$ (middle panel), and
$\epsilon=0.40$ (right panel).}%
\label{fig6}%
\end{figure}
%

\begin{figure}[h]%
\centering
\includegraphics[
height=3.0926in,
width=4.6345in
]%
{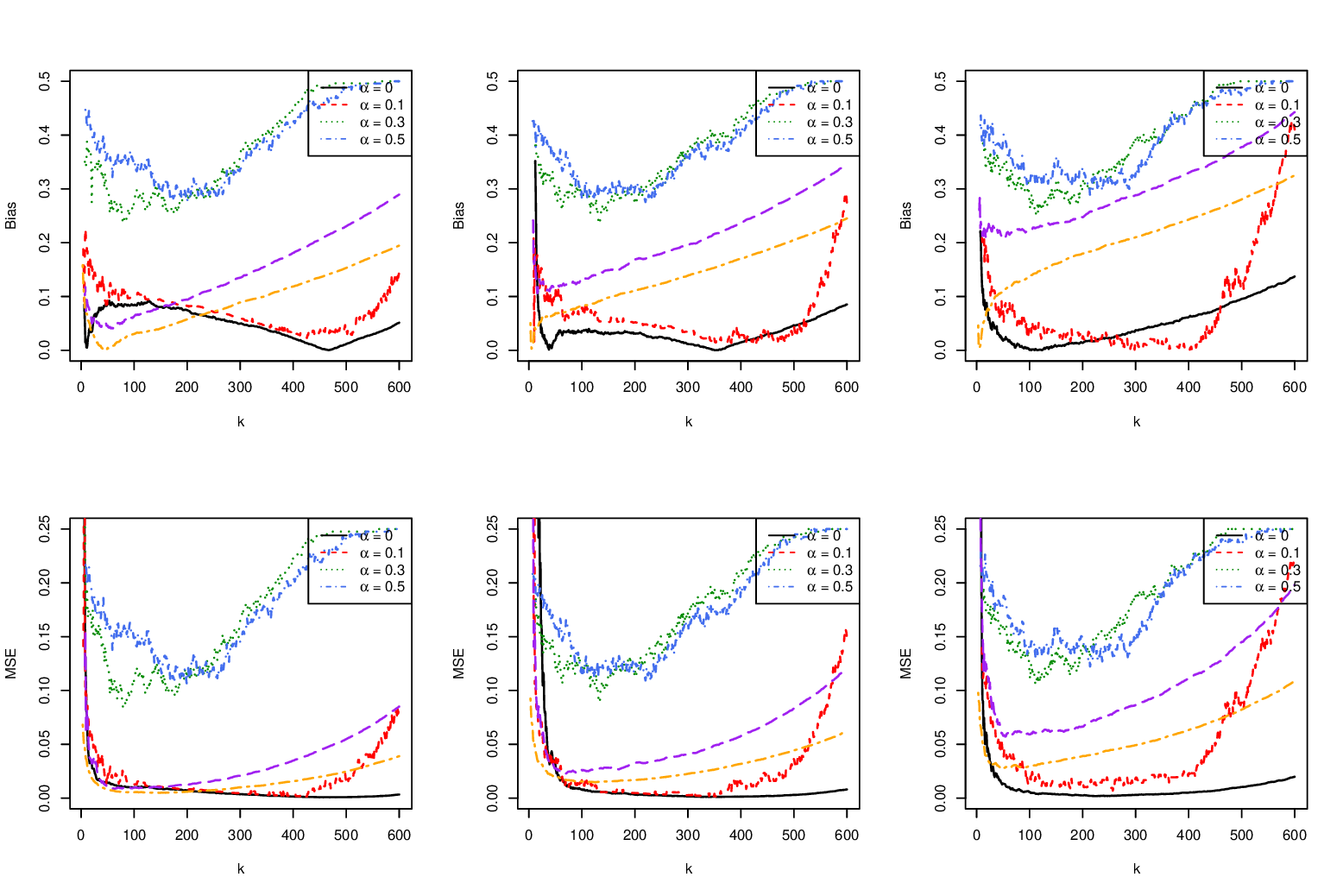}%
\caption{Bias (top panel) and MSE (bottom panel) of the MDPD tail index
estimator $\protect\widehat{\gamma}_{1,\alpha},$ together with
$\protect\widehat{\gamma}_{1}^{(EFG)}$ (long-dashed purple line) and
$\protect\widehat{\gamma}_{1}^{(W)}$\ (two-dashed orange line) computed from
$2000$ Monte Carlo samples of size $1000$ genrated from Scenario \textbf{S2},
with $\gamma_{1}=0.5,$ $\gamma_{c}=0.8$ and $p=0.55.$ Contamination is
introduced before the censoring mechanism, with $\epsilon=0$ (left panel),
$\epsilon=0.15$ (middle panel), and $\epsilon=0.40$ (right panel).}%
\label{fig7}%
\end{figure}
%

\begin{figure}[h]%
\centering
\includegraphics[
height=3.0926in,
width=4.6345in
]%
{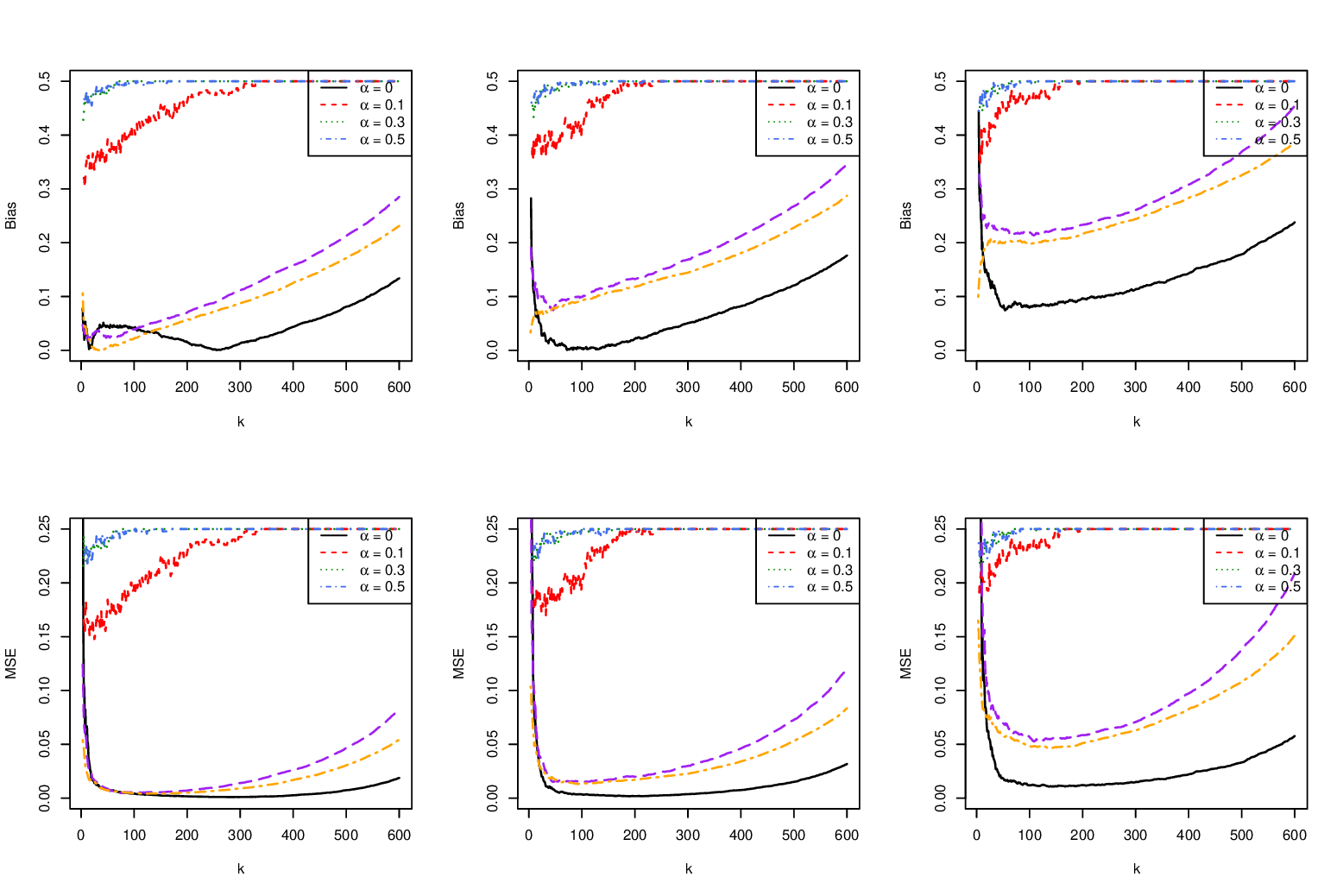}%
\caption{Bias (top panel) and MSE (bottom panel) of the MDPD tail index
estimator $\protect\widehat{\gamma}_{1,\alpha},$ together with
$\protect\widehat{\gamma}_{1}^{(EFG)}$ (long-dashed purple line) and
$\protect\widehat{\gamma}_{1}^{(W)}$\ (two-dashed orange line) computed from
$2000$ Monte Carlo samples of size $1000$ genrated from Scenario \textbf{S2},
with $\gamma_{1}=0.5,$ $\gamma_{c}=0.8$ and $p=0.7.$ Contamination is
introduced before the censoring mechanism, with $\epsilon=0$ (left panel),
$\epsilon=0.15$ (middle panel), and $\epsilon=0.40$ (right panel).}%
\label{fig8}%
\end{figure}


\begin{thebibliography}{999999999999999999999999999999999999}                                                             %


\bibitem[Basu \textit{et al.}(1998)]{Basu98}Basu, A., Harris, I. R., Hjort, N.
L., \& Jones, M. C. (1998). Robust and efficient estimation by minimizing a
density power divergence. Biometrika, 85(3), 549--559.

\bibitem[Beirlant \textit{et al.}(2016)]{BBWG2016}Beirlant, J., Bardoutsos,
A., de Wet, T., \& Gijbels, I. (2016). Bias reduced tail estimation for
censored Pareto type distributions. Statistics \& Probability Letters, 109, 78--88.

\bibitem[Beirlant \textit{et al.}(2018)]{BMV2018}Beirlant, J., Maribe, G., \&
Verster, A. (2018). Penalized bias reduction in extreme value estimation for
censored Pareto-type data, and long-tailed insurance applications. Insurance:
Mathematics and Economics, 78, 114--122.

\bibitem[Beirlant \textit{et al.}(2019)]{BWW2019}Beirlant, J., Worms, J., \&
Worms, R. (2019). Estimation of the extreme value index in a censorship
framework: Asymptotic and finite sample behavior. Journal of Statistical
Planning and Inference, 202, 31--56.

\bibitem[Brazauskas and Serfling(2000)]{BS2000}Brazauskas, V., \& Serfling, R.
(2000). Robust and efficient estimation of the tail index of a
single-parameter Pareto distribution. North American Actuarial Journal, 4(4), 12--27.

\bibitem[Colosimo \textit{et al.}(2002)]{Colosimo2002}Colosimo, E. A.,
Ferreira, F. V., Oliveira, M., \& Sousa, C. (2002). Empirical comparisons
between Kaplan--Meier and Nelson--Aalen survival function estimators. Journal
of Statistical Computation and Simulation, 72(4), 299--308.

\bibitem[Cs\"{o}rg\H{o} \textit{et al.}(1986)]{CsCsHM86}Cs\"{o}rg\H{o}, M.,
Cs\"{o}rg\H{o}, S., Horv\'{a}th, L., \& Mason, D. M. (1986). Weighted
empirical and quantile processes. Annals of Probability, 14(1), 31--85.

\bibitem[Dell'Aquila and Embrechts(2006)]{DE2006}Dell'Aquila, R., \&
Embrechts, P. (2006). Extremes and robustness: A contradiction? Financial
Markets and Portfolio Management, 20(1), 103--118.

\bibitem[Denuit \textit{et al.}(2006)]{DPV06}Denuit, M., Purcaru, O., \&
Keilegom, I. V. (2006). Bivariate Archimedean copula models for censored data
in non-life insurance. Actuar. Pract. 13, 5-32.

\bibitem[Dierchx e\textit{t al.}(2013)]{DGG13}Dierckx, G., Goegebeur, Y., \&
Guillou, A. (2013). An asymptotically unbiased minimum density power
divergence estimator for the Pareto-tail index. Journal of Multivariate
Analysis, 121, 70--86.

\bibitem[Dierckx \textit{et al.}(2021)]{DGG2021}Dierckx, G., Goegebeur, Y., \&
Guillou, A. (2021). Local robust estimation of Pareto-type tails with random
right censoring. Sankhya A, 83(1), 70--108.

\bibitem[Einmahl \textit{et al.}(2008)]{EnFG08}Einmahl, J. H. J.,
Fils-Villetard, A., \& Guillou, A. (2008). Statistics of extremes under random
censoring. Bernoulli, 14(1), 207--227.

\bibitem[Frees and Valdez(1998)]{FV98}Frees, E. W., \& Valdez, E. A. (1998).
Understanding relationships using copulas. N. Am. Actuar. J. \textit{,} 2, 1-25.

\bibitem[Ghosh and Basu(2013)]{GB2013}Ghosh, A., \& Basu, A. (2013). Robust
estimation for independent non-homogeneous observations using density power
divergence with applications to linear regression. Electronic Journal of
Statistics, 7, 2420--2456.

\bibitem[Ghosh(2017)]{Ghosh2017}Ghosh, A. (2017). Divergence based robust
estimation of the tail index through an exponential regression model.
Statistical Methods \& Applications, 26(2), 181--213.

\bibitem[Goegebeur \textit{et al.}(2014)]{GGV2014}Goegebeur, Y., Guillou, A.,
\& Verster, A. (2014). Robust and asymptotically un biased estimation of
extreme quantiles for heavy-tailed distributions. Statistics \& Probability
Letters, 87, 108--114.

\bibitem[de Haan and Ferreira(2006)]{deHF06}de Haan, L., \& Ferreira, A.
(2006). Extreme Value Theory: An Introduction. Springer.

\bibitem[de Haan and Stadtm\"{u}ller(1996)]{deHS96}de Haan, L., \&
Stadtm\"{u}ller, U. (1996). Generalized regular variation of second order.
Journal of the Australian Mathematical Society. Series A, 61(3), 381--395.

\bibitem[Hall(1982)]{Hall82}Hall, P. (1982). On some simple estimates of an
exponent of regular variation. Journal of the Royal Statistical Society:
Series B (Methodological), 44(1), 37--42.

\bibitem[Hill(1975)]{Hill75}Hill, B. M. (1975). A simple general approach to
inference about the tail of a distribution. Annals of Statistics, 3(5), 1163--1174.

\bibitem[Ju\'{a}r\`{e}z and Schucany(2004)]{JS2004}Ju\'{a}rez, S. F., \&
Schucany, W. R. (2004). Robust and efficient estimation for the generalized
Pareto distribution. Extremes, 7(3), 237--251.

\bibitem[Kaplan and Meier(1958)]{KM58}Kaplan, E. L., \& Meier, P. (1958).
Nonparametric estimation from incomplete observations. Journal of the American
Statistical Association, 53(282), 457--481.

\bibitem[Kim and Lee(2008)]{Kim2008}Kim, M., \& Lee, S. (2008). Estimation of
a tail index based on minimum density power divergence. Journal of
Multivariate Analysis, 99(10), 2453--2471.

\bibitem[Klugman and Parsa(1999)]{Klugman99}Klugman, S. A., \& Parsa, R.
(1999). Fitting bivariate loss distributions with copulas. Insurance Math.
Econom. 24, 139-148.

\bibitem[Lehmann and Casella(1998)]{LC98}Lehmann, E. L., \& Casella, G.
(1998). Theory of Point Estimation. Springer.

\bibitem[Mattys and Beirlant(2003)]{MB2003}Matthys, G., \& Beirlant, J.
(2003). Estimating the extreme value index and high quantiles with exponential
regression models. Statistica Sinica, 13(3), 853--880.

\bibitem[Mancer \textit{et al.}(2025)]{MNM2025}Mancer, S., Necir, A., \&
Meraghni, D. (2025). Robust and smooth estimation of the extreme tail index
via weighted minimum density power divergence. arXiv preprint. https://arxiv.org/abs/2507.15744

\bibitem[Meraghni \textit{et al.}(2025)]{MNS2025}Meraghni, D., Necir, A., \&
Soltane, L. (2025). Nelson-Aalen tail product-limit process and extreme value
index estimation under random censorship. Sankhya A. https://doi.org/10.1007/s13171-025-00384-y

\bibitem[Minkah \textit{el al.}(2023A)]{MWG2023A}Minkah, R., de Wet, T., \&
Ghosh, A. (2023). Robust estimation of Pareto-type tail index through an
exponential regression model. Communications in Statistics -- Theory and
Methods, 52(2), 479--498.

\bibitem[Minkah \textit{el al.}(2023B)]{MWGY2023B}Minkah, R., de Wet, T.,
Ghosh, A., \& Yousof, H. M. (2023). Robust extreme quantile estimation for
Pareto-type tails through an exponential regression model. Communications for
Statistical Applications and Methods, 30(6), 531--550.

\bibitem[Ndao \textit{et al.}(2014)]{Ndao14}Ndao, P., Diop, A., \& Dupuy, J.
F. (2014). Nonparametric estimation of the conditional tail index and extreme
quantiles under random censoring. Computational Statistics \& Data Analysis,
79, 63--79.

\bibitem[Nelson(1972)]{Nelson1972}Nelson, W. (1972). A short life test for
comparing a sample with previous accelerated test results. Technometrics,
14(1), 175--185.

\bibitem[Neves and Fraga Alves(2004)]{NevesAlves04}Neves, C., \& Fraga Alves,
M.I. (2004). Reiss and Thomas' automatic selection of the number of extremes.
Computational statistics \& data analysis, 47 (4), 689-704.

\bibitem[Reiss and Thomas(2007)]{ReTo7}Reiss, R. D., \& Thomas, M. (2007).
Statistical analysis of extreme values: With applications to insurance,
finance, hydrology and other fields (3rd ed.). Birkh\"{a}user.

\bibitem[Ripley and Solomon(1994)]{RS-94}Ripley, B. D., \& Solomon, P. J.
(1994). A note on Australian AIDS survival. Research Report 94/3, University
of Adelaide, Department of Statistics.

\bibitem[Shorack and Wellner(1986)]{SW86}Shorack, G. R., \& Wellner, J. A.
(1986). Empirical Processes with Applications to Statistics. Wiley.

\bibitem[Stupfler(2016)]{S-16}Stupfler, G. (2016). Estimating the conditional
extreme-value index under random right-censoring. Journal of Multivariate
Analysis, 144, 1--24.

\bibitem[Venables and Ripley(2002)]{VR-02}Venables, W. N., \& Ripley, B. D.
(2002). Modern Applied Statistics with S (4th ed.). Springer.

\bibitem[Worms and Worms(2014)]{WW2014}Worms, J., \& Worms, R. (2014). New
estimators of the extreme value index under random right censoring for
heavy-tailed distributions. Extremes, 17(3), 337--358.

\bibitem[Worms and Worms(2021)]{WW2021}Worms, J., \& Worms, R. (2021).
Estimation of extremes for heavy-tailed and light-tailed distributions in the
presence of random censoring. Statistics, 55(4), 979--1017.
\end{thebibliography}
\end{document}